\documentclass[11pt]{article}
\usepackage{jmlr2e_draft}

\usepackage{amsmath, amssymb} 
\usepackage{stmaryrd}
\usepackage{epsfig,calc, graphicx}
\usepackage{wasysym}
\usepackage[usenames]{color}
\usepackage{url}
\usepackage{enumitem}
\usepackage[lined,algoruled,algosection]{algorithm2e}
\usepackage{ifthen}
\usepackage{comment}
\newboolean{hideproof}
\setboolean{hideproof}{true}
\usepackage{mathtools}
\usepackage{bbm}
\usepackage{hyperref}
\usepackage{scalerel}
\usepackage{xcolor}
\usepackage{ifthen}

\newboolean{show_proof}

\usepackage{array}
\usepackage{booktabs}

\newboolean{arxiv}
\setboolean{arxiv}{true}

\newcommand\ip[2]{\langle #1, #2\rangle}
\newcommand\absip[2]{\left|\langle #1, #2\rangle\right|}

\newcommand\eps{\varepsilon}

\newcommand\dico{\Phi}

\newcommand\atom{\phi}
\newcommand\one{\mathbf{1}}

\newcommand\pdico{\Psi}

\newcommand\patom{\psi}

\newcommand\indi{\mathbbm{1}}

\newcommand\holdiag{\mathcal{E}}

\newcommand\diag{\operatorname{diag}}

\newcommand\argmax{\operatorname{argmax}}
\newcommand\argmin{\operatorname{argmin}}

\newcommand{\R}{{\mathbb{R}}}

\newcommand{\E}{{\mathbb{E}}}

\renewcommand{\P}{{\mathbb{P}}}

\newcommand{\epps}{ \varepsilon } 
\makeatletter
\newcommand*{\rom}[1]{\expandafter\@slowromancap\romannumeral #1@}
\makeatother

\newcommand{\weights}{D_{\!\scaleto{ \sqrt{\pi}\mathstrut}{6pt}}}
\newcommand{\weightss}{D_{\pi}}
\newcommand{\sqrtpi}{\scaleto{\sqrt{\pi}\mathstrut}{6pt}}
\newcommand{\coeff}{D_{\beta}}
\newcommand{\coefficient}{\beta}

\newcommand{\coeffm}{D_{\alpha }}

\DeclareFontFamily{U}{wncy}{}
\DeclareFontShape{U}{wncy}{m}{it}{<->wncyi10}{}
\DeclareSymbolFont{UWCyr}{U}{wncy}{m}{it}
\newcommand{\holb}{N}
\DeclareMathSymbol{\hol}{\mathalpha}{UWCyr}{"49}

\newcommand{\alphamin}{\underline{\alpha}}
\newcommand{\deltamin}{\delta_{\star}}
\newcommand{\pimin}{\underline{\pi}}

\newcommand{\mata}{A}
\newcommand{\matb}{B}
\newcommand{\matbs}{T}

\newcommand{\transp}{^{\ast}}

\newcommand{\RN}[1]{\uppercase\expandafter{\romannumeral#1}}

\begin{document}

\title{Convergence of alternating minimisation algorithms for dictionary learning}


\author{\name Simon Ruetz \email simon.ruetz@uibk.ac.at\\
	\name Karin Schnass \email karin.schnass@uibk.ac.at\\
	\addr  
	Universit\"at Innsbruck\\
	Technikerstra\ss e 13\\
	6020 Innsbruck, Austria}

\editor{}

\maketitle
\begin{abstract}
  In this paper we derive sufficient conditions for the convergence of two popular alternating minimisation algorithms for dictionary learning - the Method of Optimal Directions (MOD) and Online Dictionary Learning (ODL), which can also be thought of as approximative K-SVD. We show that given a well-behaved initialisation that is either within distance at most $1/\log(K)$ to the generating dictionary or has a special structure, ensuring that each element of the initialisation only points to one generating element, both algorithms will converge with geometric convergence rate to the generating dictionary. This is done even for data models with non-uniform distributions on the supports of the sparse coefficients. These allow the appearance frequency of the dictionary elements to vary heavily and thus model real data more closely.
  
\end{abstract}

\begin{keywords}
dictionary learning; sparse matrix factorisation; sparse coding; Method of Optimal Directions; MOD; Online Dictionary Learning; ODL; K-Singular Value Decomposition; K-SVD; convergence; non-uniform support distribution; rejective sampling
\end{keywords}


\section{Introduction}

The goal of dictionary learning is to find compact data representations by factorising a data matrix $Y  \in \R^{d \times N}$ into the product of a dictionary matrix $\dico \in \R^{d \times K}$ with normalised columns and a sparse coefficient matrix $X\in \R^{K \times N}$
\begin{align}\label{dict1}
    Y \approx \dico X \quad \text{and} \quad X \; \text{sparse}.
\end{align}
The sparsifying dictionary or sparse components of a data class can be used to identify further structure in the data, \cite{olsfield96}, or simply be exploited for many data processing tasks, such as signal restoration or compressed sensing, \cite{mabapo12,do06cs, carota06}.
Following the seminal Bayesian approach,~\cite{olsfield96}, there exist several strategies and algorithms to tackle the above problem~\cite{olsfield96,ahelbr06, enaahu99,sc15,kreutz03,lese00,mabaposa10,sken10,rudu12} and a growing number of theoretical results to accompany them~\cite{grsc10, spwawr12, argemo13, sc14,  bagrje14, bakest14, argemamo15,  aganjane16, aganne17, suquwr17a, suquwr17b, bach17, sc18, qu19, pasc23}. These fall mainly into two categories - optimisation based approaches and graph clustering algorithms. While graph clustering algorithms have stronger theoretical success guarantees, especially in the overcomplete case $K>d$, \cite{aganne17, argemo13}, optimisation based approaches, in particular alternating optimisation algorithms, are extremely successful and popular in practice. The common starting point of three golden classics among these alternating optimisation algorithms, the Method of Optimal Directions (MOD), \cite{enaahu99}, the K-Singular Value Decomposition (K-SVD), \cite{ahelbr06} and the algorithm for Online Dictionary Learning (ODL), \cite{mabaposa10}, is the following programme. Given the data matrix $Y= \left( y_1 , \ldots, y_N \right)$ and a dictionary size $K$ 
find a dictionary $\pdico= \left( \patom_1, \ldots , \patom_K \right)$ and sparse coefficients $X=(x_1,\ldots x_N)$ that optimise
\begin{align}
    \argmin_{\pdico, X} \| Y - \pdico X \|^2_F \quad \quad \text{s.t.}\quad X\in \mathcal{S} \quad \text{and} \quad \|\patom_k\|_2 = 1   \quad  \text{for all } k\label{dict2}.
\end{align}
The set $\mathcal{S}$ enforces sparsity on the coefficients matrix, for instance by allowing only $S$ non-zero coefficients per column, that is $\|x_n\|_0\leq S$. The norm constraint on the dictionary elements, also called atoms, removes the scaling ambiguity between the dictionary and the coefficients. Note that there remains a sign ambiguity, so for any solution to the problem above, we get $2^K$ equivalent solutions by flipping the signs of the atoms and adjusting $X$ accordingly. Since the problem is not convex there might exist even more local and/or global minima and to find a local minimum one cannot use gradient descent as the gradient with respect to $(\pdico, X)$ does not have a closed form solution. This is why the three classic algorithms from above employ
alternating optimisation and alternate between optimising the coefficients while keeping the dictionary fixed and optimising the dictionary while keeping the coefficients fixed. To update the coefficients, one aims for instance to find 
\begin{align}
    \hat{X} = \argmin_{X} \| Y - \pdico X \|^2_F = \argmin_{X} \sum_n \|y_n -\pdico x_n\|_2^2 \quad \quad \text{s.t.} \quad \|x_n\|_0 \leq S,\label{opt_coeff}
\end{align}
which corresponds to $N$ sparse approximation problems. While solving these problems exactly is NP-complete in general, \cite{na95,til15}, and thus unfeasible for large $S$, there exist many efficient routines, which have success guarantees under additional conditions and perform well in practice, such as OMP, \cite{parekr93}, used for MOD and K-SVD. LARS, \cite{efhajoti04}, used for ODL solves a related problem.
In this paper we consider simple thresholding, which is on par with computationally more involved algorithms like OMP in a dictionary learning context, \cite{parusc21}, and still relatively easy to analyse theoretically.\\
Conversely, to update the dictionary for fixed coefficients $\hat{X}$ one aims to find
\begin{align}
    \argmin_{\pdico} \| Y - \pdico \hat X \|^2_F  =:f(\pdico) \quad\quad \text{s.t.} \quad \|\patom_k\|_2 = 1 . \label{opt_dico}
\end{align}
The approach leading to the MOD update, \cite{enaahu99}, is to give up the unit norm constraint on the dictionary elements in~\eqref{opt_dico}. The modified problem then has a closed form solution, $\argmin_{\pdico}f(\pdico) = Y{\hat X}^{\dagger}$, so one can simply update the dictionary as $Y {\hat X}^{\dagger} D$, where $D$ is a diagonal matrix ensuring that each atom has unit norm. \\
Another idea is to solve~\eqref{opt_dico} approximately with one step of projected block coordinate descent. So we first calculate the gradient of $f$ with respect to $\pdico$ resulting in $$\nabla_{\pdico} f(\pdico) = - \sum (y_n - \pdico \hat x_n) \hat x_n \transp = - Y {\hat X}\transp + \pdico {\hat X} {\hat X}\transp.$$
We then choose adapted step sizes for each dictionary element stored in the diagonal matrix $\Lambda$ and -- with $D$ again a diagonal matrix ensuring normalisation -- update the dictionary as
\begin{align}
   \left[\pdico + ( Y \hat X\transp - \pdico \hat  X \hat  X\transp)\cdot\Lambda\right]\cdot D = \left[\pdico\Lambda^{-1} - Y \hat X\transp + \pdico \hat X \hat X\transp\right] \cdot \Lambda D.\label{odl_update}
\end{align} 
The particular choice $\Lambda^{-1} = \diag(\hat X \hat X\transp)$ leads to the dictionary update in ODL, \cite{mabapo12}.\\
Finally, for K-SVD the update of the $k$-th atom and corresponding coefficients can be derived from \eqref{opt_dico} or rather \eqref{dict2} by fixing $\pdico,\hat X$ except for the $k$-th atom $\patom_k$ and its corresponding locations (but not values) of non-zero entries in the $k$-th row of $\hat X$ and optimising for both using a singular value decomposition. For more details on the update we refer to \cite{ahelbr06}. However, we want to point out that K-SVD and ODL can be seen as closely related. Indeed if we take an approximate version of K-SVD, known as aK-SVD, \cite{ruziel08}, where the SVD is replaced with a 1 step power iteration, and additionally skip the coefficient update during the dictionary update, we arrive again at ODL, see 
\cite{ru22diss} for more details.
To keep the close link in mind we will refer to the dictionary update as in \eqref{odl_update} as ODL (aK-SVD) update.\\
As already mentioned the disadvantage of the practically efficient, alternating optimisation algorithms for dictionary learning is that they are less supported by theory.
In particular, when assuming that the data follows a random sparse model with an overcomplete generating dictionary $\dico$, there are no global recovery guarantees, in the sense that one of the algorithms above or a variant will recover $\dico$ with high probability. 
\begin{algorithm}[tb]  \label{dl-algos}
\caption{MOD and ODL (aK-SVD) - one iteration}  
\DontPrintSemicolon
  \SetAlgoLined
  \SetKwInOut{Input}{Input}
  \SetKwInOut{Output}{Output}
  \Input{ $\pdico,Y,S,\kappa$ (with $\kappa > 1$)}
   Set $\hat{X} = ( \hat{x}_1,\ldots ,\hat{x}_N )=0$ \;
  \ForEach{$n$}{
    $\hat{I}_n = \argmax_{I : |I| = S } \|\pdico_I\transp y_n \|_1$ \tcp*{thresholding}
    $\hat{x}_n(\hat{I}_n) \leftarrow \pdico_{\hat{I}_n}^{\dagger} y_n$ \tcp*{coefficient estimation}
    \If{$\|\hat{x}_n\|_2 \geq  \kappa \|y_n\|_2$}{
        $\hat{x}_n \leftarrow 0$ \tcp*{set pathological estimates to zero}}
        }
    $\pdico \leftarrow\begin{cases}
      Y \hat{X}^{\dagger} &\\
       Y \hat X \transp- \pdico \hat X \hat X\transp + \pdico \diag(\hat X \hat X\transp)& 
    \end{cases}  \qquad \qquad \quad \quad \: \:
     \begin{array}{r}
      \verb|// MOD update|\\
      \verb|// ODL (aK-SVD) update|\\
    \end{array}  $

    $\pdico \leftarrow (\patom_1/ \| \patom_1\|_2,\ldots ,\patom_K/ \| \patom_K\|_2 )$ \tcp*{atom normalisation }
    \Output{$\pdico$}

\end{algorithm}


\paragraph{Our Contribution}
In this paper we will cement the theoretical foundations of dictionary learning via alternating minimisation by characterising the convergence basin for both MOD and ODL (aK-SVD) in combination with thresholding as sparse approximation algorithm as summarised in Algorithm~\ref{dl-algos}.
In particular we will show that under mild conditions on the generating dictionary, any initialisation within atom-wise $\ell_2$-norm and scaled operator norm distance of order $1/\log K$ will converge exponentially close to the generating dictionary given enough training signals. Moreover, we provide further conditions on an initialisation that ensure convergence even for distances close to the maximal distance $\sqrt{2}$.\\ 
In case of ODL (aK-SVD) this is the first result of this kind. For MOD a convergence radius of order $1/S^2$ was derived in \cite{aganjane16}, however, our results show that even under milder assumptions the basin of convergence is several orders of magnitude larger. Finally, our most important contribution is that we provide the first theoretical results which are valid for data generating models where the non-zero entries of the generating coefficients are not essentially uniformly distributed. This means that some dictionary elements can be more likely to appear in the data representation than others. Such a non-homogeneous use of dictionary elements can usually be observed in dictionaries learned on image or audio data. Our results, proving the stability and convergence of alternating minimisation methods for non-uniform sparse support models, might be an explanation for their practical advantages over graph clustering or element-wise dictionary learning methods.
\paragraph{Organisation}
 The remainder of the paper is organised as follows. After introducing the necessary notation in Section~\ref{sec:notations:dict} we will define our random signal model and give some intuition why MOD and ODL should converge in Section~\ref{sec_probmodel}. We then provide our main result together with some explanations and connections to other work and the proof of the main theorem in Section~\ref{sec:main:dict}. However, to make the proof more accessible, the four lemmas it relies on are deferred to Appendix~\ref{app:4_inequalities} and further technical prerequisites to Appendix~\ref{app:technical}. We conclude with a discussion of our results and an outline of future work in Section~\ref{sec:discussion}.

\noindent

\section{Notation and setting}\label{sec:notations:dict}
The notation and setting follows~\cite{rusc22,ru22diss} very closely. Let $A \in \R^{d \times K}$ and $B \in \R^{K,m}$. Let $A_k$ and $A^k$ denote the $k$-th column and $k$-th row of $A$ respectively and $A\transp$ denote the transpose of the matrix $A$. For $1 \leq p,q,r \leq \infty$ we set $\| A \|_{p,q} := \max_{\| x\|_q = 1} \| Ax \|_p $. Recall that $\| A B \|_{p,q} \leq \|A \|_{q,r} \|B\|_{r,p}$ and $\|Ax \|_{q} \leq \|A\|_{q,p} \|x \|_{p}$. Often encountered quantities are $\| A \|_{2,1} = \max_{k \in \{1, \dots , K \}}\|A_k \|_2 $ and $\| A \|_{\infty, 2} = \max_{k \in \{1, \dots , d \}}\|A^k \|_2,$
which denote the maximal $\ell_2$ norm of a column resp. row of $A$. For ease of notation we sometimes write $\| A \| = \| A \|_{2,2}$ for the largest absolute singular value of $A$. For a vector $v \in \R^{d}$, we denote by $\underline{v} := \min_{i} |v_i|$ the smallest absolute value of $v$ and $\|v \|_{\infty}$ the maximal entry of $v$. 
For a subset $I\subseteq \mathbb{K}:=\{ 1,\dots ,K \}$, called the support, we denote by $A_I$ the submatrix with columns indexed by $I$ and $A_{I,I}$ the submatrix with rows and columns indexed by $I$.  We further set $R_I := (\mathbb{I}_I )\transp \in \R^{|I| \times K}$, allowing us to write $A_I = A R_I\transp$. This also allows us to embed a matrix $A_I\in \R^{d \times S}$ into $\R^{d \times K}$ by zero-padding via $A_I R_I \in \R^ {d \times K}$. We use the convention that subscripts take precedent over transposing, e.g. $A_I\transp = (A_I)\transp$. We denote by $\one_{I} \in \R^{K}$ the vector, whose entries indexed by $I$ are $1$ and zero else. Further, for any vector $v$ we denote by $D_v$ resp. $\diag(v)$ the diagonal matrix with $v$ on the diagonal and abbreviate
$D_{v\cdot w} : = D_v\cdot D_w$. Finally we write $\odot$ for the Hadamard Product (or pointwise product) of two matrices/vectors of the same dimension.\\
Throughout this paper we will denote by $\dico \in \R^{d \times K}$ the generating dictionary, i.e. the ground truth we want to recover, and by $\pdico \in \R^{d \times K}$ the current guess. Wlog we will assume that the columns of $\pdico$ are signed in way that the vector $\alpha \in \R^{K}$ defined via $\alpha_i := \ip{\atom_i}{\patom_i}$ has only positive entries. We will denote the $\ell_2$-distance between dictionary elements by
\begin{align}
\epps(\pdico, \dico) := \| \pdico -  \dico\|_{2,1} = \max_i \| \patom_i - \atom_i \|_2 \quad \Leftrightarrow \quad \epps(\pdico, \dico)^2 = 2-2\alphamin.
\end{align}
If it is clear from context, we will sometimes write $\epps$ instead of $\epps(\pdico,\dico)$. An important variable which will be used frequently throughout this paper is $Z: = \pdico - \dico$ --- the difference matrix between the generating dictionary $\dico$ and a guess $\pdico$. We define the distance between $\dico $ and $\pdico$ as
\begin{align}
    \delta(\pdico,\dico) := \max \Big\{ \|(\pdico - \dico)\weights \|, \| \pdico -  \dico\|_{2,1} \Big\} = \max \Big\{ \|Z\weights \|, \| Z\|_{2,1} \Big\}.
\end{align}
This might not seem intuitive at first glance, but to show convergence we have to control the weighted operator norm of the difference matrix as well as the $\ell_2$-distance. Again, if it is clear from context we will simply write $\delta$.  

\section{Probabilistic model}\label{sec_probmodel}
As was noted in the introduction we want the locations of the non-zero coefficients to follow a non-uniform distribution, allowing some dictionary elements to be picked more frequently than others. 
\begin{definition}[Poisson and rejective sampling]\label{def_prob_models}
Let $\delta_k$ denote a sequence of $K$ independent Bernoulli $0$-$1$ random variables with expectations $0 \leq p_k \leq 1$ such that $\sum_{k =1}^K p_k = S$ and denote by $\P_B$ the probability measure of the corresponding Poisson sampling model. 
We say the support $I$ follows the Poisson sampling model, if $I  := \left\{ k \; \middle| \; \delta_k  = 1\right\}$ and each support $I \subseteq \mathbb{K}$ is chosen with probability
\begin{equation}\label{ind_dist:dict}
    \P_B(I)=\prod_{i \in I}p_i\prod_{j \notin I}(1-p_j).
\end{equation}
We say our support $I$ follows the rejective sampling model, if each support $I \subseteq \mathbb{K}$ is chosen with probability
\begin{equation}\label{cond_dist:dict}
    \P_S(I) := \P_B(I \; | \;  |I| = S). 
\end{equation}
If it is clear from the context, we write $\P(I)$ instead of $\P_S(I)$.
\end{definition}
The Poisson sampling model would be quite convenient since the probability of one atom appearing in the support is independent of the others. Unfortunately, we need a model with exactly S-sparse supports, rather than supports that are S-sparse on average. The rejective model satisfies this second condition and still yields almost independent atoms, in the sense that we can often reduce estimates for rejective sampling to estimates for Poisson sampling, \cite{rusc22}. Note also that unless we are in the uniform case, where $p_i = S/K$ for all $i$, the inclusion probabilities $\pi_i := \P(i\in I)$ are different from the parameters $p_i$. They can be related using \cite{rusc22}[Lemma 1], which we restate for convenience in Appendix~\ref{sec:supp_rand_mat}. We frequently use the square diagonal matrix $\weights := \diag((\sqrt{\pi_k})_k)$. Based on the rejective sampling model, we define the following model for our signals.
\begin{definition}[Signal model]\label{signal_model:dict}
Given a generating dictionary $\dico \in \R^{d \times K}$ consisting of $K$ normalized atoms, we model our signals as
\begin{align}
    y = \dico_I x_I  =  \sum_{i \in I} \atom_{i} x_{i}   , \quad x_{i} = c_{i} \sigma_{i}, 
\end{align}
where the support $I = \{i_1, \dots i_S \} \subseteq\mathbb{K}$ is chosen according to the rejective sampling model~\eqref{cond_dist:dict} with parameters $p_1, \dots , p_K$ such that $\sum_{i = 1}^K p_i = S$ and $0 < p_k \leq 1/6$, the coefficient sequence $c = (c_i)_i \in \R^K$ consists of i.i.d. bounded random variables $c_i$ with $0 \leq  c_{\min} \leq  c_i \leq c_{\max} \leq 1$ and the sign sequence $\sigma \in \{-1,1\}^K$ is a Rademacher sequence, i.e. its entries $\sigma_i$ are i.i.d with $\P(\sigma_i =\pm 1)= 1/2$. Supports, coefficients and signs are modeled as independent and we can write
$x = \one_I \odot c \odot \sigma$.
\end{definition}
The assumption $p_i \leq 1/6$ ensures that the $p_i$ and the corresponding inclusion probabilities of the rejective sampling model $\pi_i$ are not too different (see Theorem~\ref{th_poisson}). We further introduce the vector $\beta \in \R^K$ via $\beta_i := \E [ c_i^2 ]$ and denote by $D_{\coefficient}$ the corresponding diagonal matrix. To see how this signal model allows us to prove convergence of the algorithms, note that before normalisation, by using $Y = \dico X$, we can write the two dictionary update steps concisely as 
\begin{align*}
 \text{MOD:} \quad   & Y \hat{X} \transp (\hat{X} \hat{X}\transp)^{-1} = \dico X \hat{X} \transp (\hat{X} \hat{X}\transp)^{-1},\\
 \text{ODL:} \quad  &  \frac{1}{N} \left[Y \hat X \transp -\pdico \hat X \hat X\transp + \pdico \diag(\hat X \hat X\transp) \right] =  \frac{1}{N} \left[\dico X \hat X \transp - \pdico \hat X \hat X\transp + \pdico \diag(\hat X \hat X\transp) \right],
\end{align*}
where we scaled the update step of the ODL algorithm by a factor of $1/N$. The key is to show that both these update steps concentrate around $\dico$. To that end we define the two averages of random matrices
\begin{align}
    \mata:= \frac{1}{N} X \hat X\transp = \frac{1}{N}\sum_{n = 1}^N x_n \hat{x}_n \transp 
\quad \text{and} \quad 
    \matb := \frac{1}{N} \hat X \hat X\transp = \frac{1}{N} \sum_{n = 1}^N \hat{x}_n \hat{x}_n \transp. 
\end{align}
With these we can write the update step of MOD as $\dico A B^{-1}$ and that of ODL as $\dico \mata - \pdico [\matb -\diag(\matb)]$. We first take a closer look at the terms within the sums above, where for simplicity we drop the index $n$. Assuming that thresholding finds the correct support $I$, we can write $x\hat{x}\transp$ using the zero-padding operator $R_I\transp$ as
\[
x \hat{x} \transp = x (R_I\transp \pdico_I^\dagger y)\transp = x x\transp \dico \transp(\pdico_I^{\dagger})\transp R_I ,
\]
 Further assuming that $\pdico_I$ is well conditioned, meaning $\pdico_I \transp \pdico_I \approx \mathbb{I}$, we can approximate $\pdico_I^\dagger \approx \pdico_I\transp$, leading to
\[
x \hat{x} \transp \approx x x \transp \dico \transp\pdico_I R_I = x x \transp \dico \transp\pdico R_I\transp R_I = x x \transp \dico \transp\pdico \diag(\one_I) .
\]
As we modelled the generating coefficients as $x = c \odot \sigma \odot \one_1$, using the independence of $c,\sigma, I$ we get that in expectation over $c,\sigma$, 
$$\E_{c,\sigma}[x x \transp ] = \E_c[c c\transp] \odot \E_\sigma[\sigma \sigma \transp] \odot (\one_I \one_I\transp) = \coeff  \diag(\one_I) = \diag(\one_I)\coeff.$$ 
So the empirical estimator $\mata=\frac{1}{N}\sum_{n = 1}^N x_n \hat{x}_n \transp\approx\E[x\hat x]$ will be well approximated by
\[
\E[x \hat{x} \transp] \approx \E  \diag(\one_I) \coeff \dico \transp\pdico \diag(\one_I) = (\coeff \dico \transp \pdico ) \odot \E[\one_I \one_I \transp].
\]
The matrix $\E[\one_I \one_I \transp]$ simply stores as $ij$-th entry how often $\{i,j\}\subseteq I$, meaning the diagonal entries are far larger than the off-diagonal ones, and we can approximate $\E[\one_I \one_I \transp]\approx D_{\pi} + \pi \pi \transp \approx D_{\pi}$.
Finally using that $\coeffm = \diag(\dico \transp \pdico)$ a similar analysis for $\matb$ yields
\begin{align*}
\mata \approx \E[x \hat{x} \transp] \approx (\coeff \dico \transp \pdico )\odot D_{\pi} = D_{\pi \cdot \alpha \cdot \coefficient} \quad \text{and} \quad \matb \approx \E[\hat x \hat{x} \transp]  \approx D_{\pi \cdot \alpha^2 \cdot \coefficient}.
\end{align*}
So before normalisation the updates via MOD and ODL should be approximately
\begin{align}
\dico A B^{-1} \approx \dico \coeffm^{-1} 
\quad \text{and} \quad
\dico \mata - \pdico [\matb -\diag(\matb)] \approx \dico D_{\pi \cdot \alpha \cdot \coefficient}. \label{heuristics}
\end{align}
This means that the output of both dictionary update steps after normalisation should be very close to the ground truth, and the proof boils down to quantifying the error in the approximation steps outlined above. 




\section{Main result}\label{sec:main:dict}
Concretely, we will prove the following theorem.
\begin{theorem} \label{the_theorem}
Assume our signals follow the signal model in~\ref{signal_model:dict}.
Define 
\[
\alphamin := \min_{k} \absip{\patom_k}{\atom_k} = 1 - \epps^2/2,  \quad  \gamma : =\frac{c_{\min}}{c_{\max}} \quad \text{and} \quad \rho := 2 \kappa^2 S^2 \gamma^{-2} \alphamin^{-2}\pimin^{-3/2}  ,
\]
with $\kappa^2\geq2$ and let $C,n$ be two universal constants, where $C$ is no larger than $42$ and $n$ no larger than $130$. Denote by $\deltamin$ the desired recovery accuracy and assume $\deltamin \log\left(n K \rho / \deltamin \right) \leq \gamma / C $. We abbreviate $\nu = 1/\sqrt{\log\left(n K \rho / \deltamin \right)}<1$. If the atom-wise distance $\varepsilon = \varepsilon(\pdico,\dico)$ of the current guess $\pdico$ to the generating dictionary $\dico$ satisfies
\begin{align}
 \max\big \{  \nu \|\dico \weights \|, \mu(\dico)\big\} \leq \left( 1 - \frac{\epps^2}{2}\right) \cdot  \frac{\gamma}{4C \log\left( n K \rho / \deltamin \right)} \label{cond_dico}
\end{align}
and the current guess $\pdico$ additionally satisfies either
\begin{align}
 \max\big \{  \nu \|\pdico \weights \|, \mu(\pdico),\mu(\pdico,\dico) \big\}\leq  \left( 1 - \frac{\epps^2}{2}\right) \cdot \frac{ \gamma}{4C\log\left( n K \rho / \deltamin \right)} \label{regime1}
\end{align}
\begin{align}
& \text{or}\qquad \qquad   \delta(\pdico,\dico) \leq \frac{\gamma}{C\log\left( n K \rho / \deltamin \right)} =: \delta_{\circ},\label{regime2}\qquad
\end{align}
 then the updated and normalised dictionary $\hat{\pdico}$, which is output by ODL or MOD, satisfies
\begin{equation}
\delta(\hat{\pdico},\dico) \leq  \frac{1}{2} \cdot \big(\deltamin/2+  \min \left\{ \delta_{\circ}, \delta(\pdico,\dico) \right\}\big)=: \frac{1}{2} \cdot  \Delta ,\label{contraction}
\end{equation}
except with probability
\begin{equation}
60 K \exp{\left(- \frac{N (  \Delta/16)^2 }{2 \rho^2+ \rho   \Delta/16 }\right)}.\label{total_failure_bound}
\end{equation}
\end{theorem}
To make the theorem more accessible, we now provide a detailed discussion of the implications and assumptions.
\begin{description}[wide, labelwidth=0pt, labelindent=0pt]
\item[Convergence] First note that a repeated application of the theorem proves convergence of MOD and ODL (aK-SVD) up to accuracy $\deltamin$, assuming a generating dictionary and an initialisation that satisfies the conditions above and given a new batch of training signal in each iteration of size which scales at worst as $N \approx \log K \rho^2/\deltamin^2$. 
To see this, note that the theorem ensures that the distance between the generating dictionary and the output dictionary decreases fast enough for condition \eqref{regime2} to be satisfied after one step.
Once an input dictionary satisfies $\deltamin \leq \delta( \pdico,\dico) \leq \delta_\circ$ each iteration will w.h.p. shrink the distance by a factor $\eta \leq 3/4$, since we have $$\delta(\hat \pdico,\dico) \leq \tfrac{1}{4} \cdot \deltamin + \tfrac{1}{2} \cdot \delta( \pdico,\dico) \leq \tfrac{3}{4} \cdot \delta( \pdico,\dico).$$
As the conditions on the generating dictionary get easier to satisfy in each step, we can iterate this argument to get convergence.\\
Finally, if an input dictionary already satisfies $\delta( \pdico,\dico) \leq \deltamin$, that is, we reached the desired precision, the contraction property \eqref{contraction} guarantees that the output dictionary still satisfies $\delta(\hat \pdico,\dico) \leq \tfrac{3}{4} \deltamin \leq \deltamin$. This stability is due to the assumption that the target recovery accuracy satisfies $\deltamin \log\left(n K \rho / \deltamin \right) \leq \gamma /C$ meaning $\deltamin \leq \delta_\circ$. Going through the proof shows that without this assumption, for instance because there are not enough training signals for a higher accuracy available, the theorem above is still true and the target distance will be achieved after one step. In this case however, since the output dictionary might no longer satisfy \eqref{regime1}, yet not be close enough to satisfy \eqref{regime2}, it is not guaranteed that a further iteration will preserve this distance.

\item[Generating dictionary]
The first condition in \eqref{cond_dico} contains as implicit conditions on the generating dictionary that 
\begin{align} \|\dico \weights\| \lesssim \frac{1}{\sqrt{\log K}}\quad \text{and} \quad \mu(\dico) \lesssim \frac{1}{\log K}, \label{cond_gen_dico}
\end{align} which together ensure that $\dico$ is a sensible dictionary. It guarantees that most random sparse supports are well-conditioned, that is $\| \dico_I \transp \dico_I - \mathbbm{I}\| \leq \vartheta < 1 $. This further means that most signals generated by the model have a stable representation, $y = \dico_I x_I$ with $\|y\|_2 \approx \|x_I\|_2$, and could be identified by a sparse approximation algorithm that has access to $\dico$.
To see that the condition is quite mild, note that for uniformly distributed supports, $p_i = \pi_i = S/K$, and $\dico$ a unit norm tight frame, $\|\dico\| = K/d$, we can rewrite it as $S\leq d/\log K$, which is for instance a common requirement in compressed sensing.

\item[Initialisation (input dictionary):] 
The condition in \eqref{cond_dico} further limits the maximal atom-wise distance of an initialisation. Assuming again uniformly distributed supports and a tight generating dictionary it corresponds to
\[
\epps(\pdico, \dico) = \max_k \| \patom_k -  \atom_k \| \lesssim \left( 2 -  2\sqrt{\frac{S \log K}{d}} \right)^{1/2}
\]
Considering that the maximal distance between two unit norm vectors is $\sqrt{2}$, this means that convergence is possible even from far away initialisations. However, to actually converge we have additional requirements. The initialisation needs to be well behaved, which is ensured by \eqref{regime1} via $\|\pdico \weights\| \lesssim \frac{1}{\sqrt{\log K}}$ and $\mu(\pdico) \lesssim \frac{1}{\log K}$, similar to \eqref{cond_gen_dico} for the generating dictionary. The more stringent and interesting requirement in \eqref{regime1} is the condition on the cross-coherence $\mu(\pdico,\dico)$ or rather the cross-Gram matrix $\pdico\transp\dico$. As the minimal entry on its diagonal is $\alphamin = (1-\eps^2/2)$ it translates to
\[
\max_{i \neq j} \absip{\patom_i}{\atom_j} \cdot \log K  \lesssim \min_k \absip{\patom_k}{\atom_k},
\]
or $\pdico\transp\dico$ being diagonally dominant. Intuitively, \eqref{cond_dico} and \eqref{regime1} mean that the admissible distance can be very close to $\sqrt{2}$, as long as the initialisation is a well-behaved dictionary and no two estimated atoms point to the same generating atom, meaning it is clear to the sparse approximation algorithm which estimated atom belongs to which generating atom. While it might be possible to relax this separation condition, it is unlikely that it can be removed in the overcomplete case. Indeed, following the guidelines in \cite{pasc23}, one can construct well-behaved incoherent initialisations for which both MOD and ODL converge to a local minimum that is not equivalent to the generating dictionary.\\
Finally, if the initialisation is within distance $1/\log K$ to the generating dictionary, meaning condition \eqref{regime2} is satisfied, it is automatically well-behaved and has a diagonally dominant cross-Gram matrix, because it inherits these properties from the generating dictionary, which by \eqref{cond_dico} is well behaved and incoherent.

\item[Number of signals:]
From the probability bound in~\eqref{total_failure_bound} we see that in order for the failure probability in each step to be small, the number of the fresh signals per iteration has to be approximately
\[
N \approx \frac{\rho^2}{\deltamin^2} \cdot \log K \approx \frac{1}{\pimin^{3} \deltamin^2} \cdot \log K.
\]
This ensures that even the most rarely appearing atoms are seen often enough to learn them properly. The relation above
reflects the general dependencies, meaning we need more training signals for higher accuracy and more imbalanced atom distributions, but is a little too pessimistic. We expect that the scaling can be reduced to $N \approx \log K \deltamin^{-1} \pimin^{-3/2}$ and even $N\approx K \log K/\deltamin$ in the uniform case, but leave the endeavour to those still motivated after reading the current proof\footnote{Hints how to proceed can be found after the proof Lemma~\ref{lemma_claim4}}.

\item[Attainable accuracy:] 
We have already seen that the target accuracy should not be chosen too large in order to have stable convergence and that the number of training signals required in each iteration should grow with the desired accuracy. Another interesting observation is that even given an arbitrarily close initialisation and arbitrarily many training signals, the best attainable accuracy is limited by the coherence and conditioning of the generating dictionary via \eqref{cond_dico} --- even in the noiseless case considered here. 
The main reason is that~\eqref{cond_dico} is a very light condition which does not exclude the existence and rare selection of sparse supports, leading to ill-conditioned or even rank-deficient matrices $\pdico_I$. Such supports cannot be recovered by thresholding or any other sparse approximation algorithm. This failure probability stops the algorithm from attaining arbitrarily small precision. If we exchange~\eqref{cond_dico} by the more restrictive assumption $2S \mu(\dico) <1$ then all supports of size $S$ are well enough conditioned to be identified by a sparse algorithm like OMP using $\dico$ (or some very small perturbation of it). This means that $\dico$ is a fixed point of the algorithm and that for a close enough initialisation we have convergence to $\dico$ in expectation.

\item[Comparison to existing results for MOD:] Finally, we can compare our result with that in \cite{aganjane16} for MOD with an $\ell_1$-minimisation based sparse approximation routine. Assuming a signal model with uniformly selected random supports and a non-zero coefficient distribution, which is bounded from above, it is shown that for an approximately tight dictionary with coherence $\mu(\dico) \lesssim 1/\sqrt{d}$ and a sparsity level $S\lesssim d^{1/6}$ MOD will converge from any initialisation satisfying $\eps(\pdico,\dico)\lesssim 1/S^2$, given enough training signals.
In this special case our assumptions on the generating dictionary --- $\mu(\dico) \lesssim 1/\log K$ and $S \lesssim d/\log(K)$ --- are more relaxed. As already seen, this comes at the price of a theoretical limit on the achievable accuracy, which in practice, however, is determined by the number of available signals. Further, we also have lighter conditions on the distance between the initialisation and the generating dictionary. In particular, for a tight dictionary our conditions can be written as
\begin{align*}
    \eps(\pdico,\dico)\lesssim \frac{1}{\log K}\quad \text{and} \quad \| \pdico -\dico \| \lesssim \sqrt{\frac{K}{S\log^2 K}},
\end{align*}
whereas the assumption in~\cite{aganjane16} that $\eps(\pdico,\dico)\lesssim 1/S^2$ implies $\|\pdico - \dico\| \lesssim \sqrt{K}/S^2$ by using $\|A\|\leq \|A\|_F\leq \sqrt{K} \|A\|_{2,1}$. Hence the assumptions in~\cite{aganjane16} are more restrictive.\\
These restrictions seem due to using $\ell_1$-minimisation for sparse approximation, which on the other hand has the advantage that it does not require the distribution of the non-zero coefficients to be bounded away from zero. However, such an assumption is used in the graph clustering algorithm suggested in \cite{aganjane16} to get an initialisation, that satisfies the convergence conditions. This indicates that such a condition is necessary to get a larger convergence area.
\end{description}

\begin{proof}[Proof of Theorem~\ref{the_theorem}]
We first collect the results of Lemmas~\ref{lemma_claim1}-\ref{lemma_claim4}. Writing 
\begin{align}
T:=(D_{\sqrtpi \cdot \alpha} )^{-1} \matb (D_{\sqrtpi \cdot \alpha \cdot \coefficient})^{-1}\quad \text{and} \quad \mathbb{I}_{\ell^c} : = \mathbb{I} -e_\ell e_\ell\transp
\end{align}
for convenience, we get that except with failure probability as in~\eqref{total_failure_bound} we have 
\begin{align*}
    \| \dico \mata (D_{\sqrtpi \cdot \alpha \cdot \coefficient})^{-1} - \dico \weights \|_{2,2} \leq  \alphamin  \Delta /8 \qquad  \text{and} \qquad
     \| \matbs - \mathbbm{I}\|_{2,2} \leq \Delta/4,
\end{align*}
as well as for all $\ell \in \{1, \cdots, K\}$ 
\begin{align*}
   \| \dico \mata (D_{\pi \cdot \alpha \cdot \coefficient})^{-1} e_{\ell} - \atom_{\ell}  \|_2 & \leq  \frac{\Delta}{8} \quad \text{and} \quad 
   \max\left\{ \frac{\gamma \alphamin \nu}{4C} , \| \pdico \weights\| \right\} \cdot \| \mathbb{I}_{\ell^c}\matbs \, e_{\ell} \pi_{\ell}^{-1/2} \|_2  \leq  \frac{ \Delta}{8}.
\end{align*} 

Using these 4 inequalities we show that for both algorithms a properly scaled version of the updated dictionary, which we denote by $\bar{\pdico}$, contracts towards the generating dictionary. Concretely, we show that for some constants $s_\ell$ close to 1
\begin{align}\label{heuristics:1}
  \| ( \bar{\pdico} - \dico)  \weights \| \leq  \Delta/4 \quad \text{and} \quad  \max_\ell \| \bar{\patom}_\ell - s_\ell \atom_\ell  \| \leq   \Delta/3.
\end{align}
Together these bounds guarantee that the normalised version of the updated dictionary $\hat \pdico$ satisfies $\delta(\hat \pdico, \dico) \leq  \Delta/2 $, meaning we have contraction towards the generating dictionary in the weighted operator norm and the maximum column norm simultaneously. We start with the proof of ODL which is a bit simpler.

\paragraph{ODL:}
Motivated by~\ref{heuristics}, we define a scaled version of the updated dictionary $\bar{\pdico } :=  \left[\dico \mata - \pdico \matb  + \pdico \diag(\matb)\right] (D_{\pi \cdot \alpha \cdot \coefficient} )^{-1},$
which ensures that on average $\bar{\pdico}$ concentrates around $\dico$. The scaling does not change the underlying algorithm, since we have a normalisation step at the end of each iteration, which we will analyse afterwards. We decompose $\bar \pdico$ as
\begin{align}
    \bar{\pdico } &= \dico \mata (D_{\pi \cdot \alpha \cdot \coefficient} )^{-1} - \pdico [\matb- \diag(\matb)](D_{\pi \cdot \alpha \cdot \coefficient})^{-1} \nonumber \\
    & =
    \dico \mata (D_{\pi \cdot \alpha \cdot \coefficient} )^{-1} - \pdico D_{\sqrtpi \cdot \alpha} \left[(D_{\sqrtpi \cdot \alpha})^{-1} \matb (D_{\pi \cdot \alpha \cdot \coefficient})^{-1}  - \weights^{-1}\right] \nonumber \\
    & \hspace{4cm} + \pdico D_{\sqrtpi \cdot \alpha} \left[(D_{\sqrtpi \cdot \alpha})^{-1} \diag(\matb)(D_{\pi \cdot \alpha \cdot \coefficient} )^{-1}- \weights^{-1}\right]\notag\\
    & =  \dico \mata (D_{\pi \cdot \alpha \cdot \coefficient} )^{-1} - \pdico D_{\sqrtpi \cdot \alpha} \left(\matbs - \mathbbm{I}\right) \weights^{-1} + \pdico D_{\sqrtpi \cdot \alpha} \left(\diag(\matbs)- \mathbbm{I}\right)\weights^{-1}. \label{odl_2norm_split}
\end{align}
We first show contraction in the weighted operator norm. Note that in both regimes we have $\|\pdico \weights\|\leq 2/C$, either by direct assumption or based on the bound 
$$\|\pdico \weights\|\leq \|\dico \weights\| + \|(\pdico - \dico) \weights\| \leq \|\dico \weights\| +\delta \leq \|\dico \weights\| +\delta_\circ.$$ With the expression for the updated dictionary $\bar{\pdico}$ in \eqref{odl_2norm_split} and using the fact that $\|D_\alpha\| \leq 1$ we can bound the operator norm of the difference $(\bar{\pdico}- \dico)\weights$ as
\begin{align*}
    \|(\bar{\pdico } - \dico) \weights\| & \leq \underbrace{\| \dico \mata (D_{\sqrtpi \cdot \alpha \cdot \coefficient })^{-1} - \dico \weights \|}_\text{\clap{$ \leq   \alphamin  \Delta/8$~\eqref{lemma_claim1}}}  +   \underbrace{2 \| \pdico D_{\sqrtpi \cdot \alpha} \|}_\text{\clap{$\leq 4 /C $}}  \underbrace{\cdot\: \| \matbs - \mathbbm{I}\|\:}_\text{\clap{$\cdot   \Delta/4$~\eqref{lemma_claim2}}} \leq \frac{ \Delta}{4}.
\end{align*}
Next we show that for each atom of the scaled dictionary the $\ell_2$-distance also decreases with each iteration. We access the $\ell$-th dictionary atom~$\bar\patom_\ell$ simply by multiplying $\bar{\pdico}$ with the standard basis vector~$e_{\ell}$. This yields
\begin{align}
   \bar\patom_\ell= \bar{\pdico } e_{\ell}&=\dico \mata (D_{\pi \cdot \alpha \cdot \coefficient} )^{-1}e_{\ell} - \pdico D_{\sqrtpi \cdot \alpha} \left[\matbs -\diag(\matbs)\right]\weights^{-1}e_{\ell} \nonumber \\
    & =
    \dico \mata (D_{\pi \cdot \alpha \cdot \coefficient})^{-1} e_{\ell} + \pdico D_{\sqrtpi \cdot \alpha} \, \mathbb{I}_{\ell^c}  \matbs\, e_{\ell} \pi_{\ell}^{-1/2}.
\end{align}
Using this decomposition together with our second set of inequalities we get
\begin{align*}
    \| \bar\patom_\ell-  \atom_{\ell} \| \leq \underbrace{\| \dico \mata (D_{\pi \cdot \alpha \cdot \coefficient})^{-1}e_\ell - \atom_{\ell}  \|}_\text{\clap{$\leq    \Delta/8$~\eqref{lemma_claim3}}} + \underbrace{\|\pdico D_{\sqrtpi \cdot \alpha  } \| \cdot\| \mathbb{I}_{\ell^c}\matbs e_{\ell}\pi_{\ell}^{-1/2}\|}_\text{\clap{$\leq   \Delta/ 8$~\eqref{lemma_claim4}}} \leq \frac{ \Delta}{4},
\end{align*}
which shows the second part in \eqref{heuristics:1} for $s_\ell =1$. To finish the proof we still need to show that \eqref{heuristics:1} guarantees contraction in the weighted operator norm and in the maximum column norm after normalisation. However, we postpone the analysis of the normalising step to after the analysis of the MOD algorithm, since it is the same for both algorithms.

\paragraph{MOD:} 
Turning to the MOD algorithm we recall that if the estimated coefficient matrix $\hat X$ has full row rank $K$ or equivalently $\hat X \hat X\transp$ has full rank, which is guaranteed by the second of our 4 inequalities, we can write the dictionary update step before normalisation as $\dico X\hat X\transp \big(\hat X \hat X\transp\big)^{-1} = \dico A B^{-1}$. This dictionary update step --- though conceptually very easy --- is harder to analyse theoretically due to the inverse of the matrix $\matb$. Again we will do the analysis for a scaled version of the updated dictionary $\Bar{\pdico} := \dico \mata \matb^{-1} \coeffm \approx \dico$.
As for ODL we start by showing that the weighted operator norm of the difference $\bar{\pdico}- \dico$ contracts.
We split $\bar{\pdico } \weights$ as follows
\begin{align}
\bar{\pdico} \weights &= \dico \mata \matb^{-1}  D_{\sqrtpi \cdot \alpha}  \nonumber\\ 
&= \dico \mata (D_{\sqrtpi \cdot \alpha \cdot \coefficient})^{-1} +\dico \mata (D_{\sqrtpi \cdot \alpha \cdot \coefficient})^{-1} \left( \left[(D_{\sqrtpi \cdot \alpha})^{-1} \matb (D_{\sqrtpi \cdot \alpha \cdot \coefficient})^{-1}\right]^{-1} - \mathbbm{I}\right) \notag\\
&=\dico \mata (D_{\sqrtpi \cdot \alpha \cdot \coefficient})^{-1} +\dico \mata (D_{\sqrtpi \cdot \alpha \cdot \coefficient})^{-1} \left( \matbs^{-1} - \mathbbm{I}\right) \label{eq:1}.
\end{align} 
Since by Lemma~\ref{lemma_claim2} the matrix $\matbs$ is close to the identity, $\| \matbs- \mathbbm{I}\| \leq \Delta/4$, the Neumann series for its inverse converges and we have $\matbs^{-1} = [\mathbbm{I} - (\mathbbm{I}-\matbs)]^{-1}  = \sum_{k\geq 0} ( \mathbbm{I}-\matbs )^k$. Using the geometric series formula we get for the operator norms of $\matbs^{-1},\matbs^{-1} - \mathbbm{I}$ 
\begin{align}
\|\matbs^{-1}\| &= \|\sum_{k\geq 0}( \mathbbm{I}-\matbs )^k\| \leq \sum_{k\geq 0}\|\mathbbm{I}-\matbs \|^k \leq \frac{1}{1-\Delta/4}, \label{T_inv_bound}\\
 \|\matbs^{-1} - \mathbbm{I} \| &= \|\sum_{k \geq 1}(\mathbbm{I}-\matbs)^k\| \leq \sum_{k \geq 1} \|\matbs - \mathbbm{I}\|^k  \leq \frac{\Delta}{4 - \Delta}\leq \frac{ \Delta}{3}, \label{T_inv_I_bound}
 \end{align}
where for the last inequality we have used that $ \Delta \leq \tfrac{3}{2} \cdot \delta_\circ  \leq 1/4 \leq 1$ since $\delta_\circ = \gamma \nu^2/C$ and $\nu \leq 1/3$. Note also that by Lemma~\ref{lemma_claim1} and \eqref{cond_dico}, we get
\begin{align}
\| \dico \mata (D_{\sqrtpi \cdot \alpha \cdot \coefficient})^{-1} \| &\leq \|  \dico \mata (D_{\sqrtpi \cdot \alpha \cdot \coefficient})^{-1} - \dico \weights   + \dico \weights \|  \leq \frac{\alphamin  \Delta}{8}  +\|\dico \weights \| \leq \frac{5}{4}  \cdot \frac{\alphamin \gamma \nu }{4 C}. \label{eq:7} 
\end{align}
Combining these observations and using the triangle inequality repeatedly yields
\begin{align*}
 \| (\bar{\pdico}   - \dico) \weights \| \leq \underbrace{\| \dico \mata  (D_{\sqrtpi \cdot \alpha \cdot \coefficient})^{-1}  - \dico\weights \|}_\text{\clap{$\leq  \alphamin  \Delta/8$~\eqref{lemma_claim1}}} + \underbrace{\| \dico \mata  (D_{\sqrtpi \cdot \alpha \cdot \coefficient})^{-1} \|}_\text{\clap{$\leq  \alphamin  \gamma \nu/C$~\eqref{eq:7}}} \cdot \underbrace{\|\matbs^{-1}  - \mathbbm{I} \|}_\text{\clap{\quad \quad $\leq  \Delta/3$~\eqref{T_inv_I_bound}}} \leq \frac{ \Delta}{4}.
\end{align*}
This shows that under the assumptions of the theorem, the weighted operator norm of the distance between the generating dictionary and the scaled update decreases in each iteration. \\
Now to the contraction of the atomwise $\ell_2$-norm. First we again split our updated dictionary atom $\bar \patom_\ell$ into two parts using that $\mathbb{I}=e_\ell e_\ell \transp + \mathbb{I}_{\ell^c}$
\begin{align}
  \bar \patom_\ell  = \bar{\pdico }e_\ell &= \dico \mata (D_{\sqrtpi \cdot \alpha \cdot \coefficient})^{-1} \matbs^{-1} e_{\ell} \pi_{\ell}^{-\frac{1}{2}}\nonumber \\
   &= \dico \mata (D_{\pi \cdot \alpha \cdot \coefficient})^{-1}  e_{\ell} \cdot e_{\ell}\transp \matbs^{-1} e_{\ell}
+ \dico \mata (D_{\sqrtpi \cdot \alpha \cdot \coefficient})^{-1} \mathbb{I}_{\ell^c} \matbs^{-1} e_{\ell} \pi_{\ell}^{-\frac{1}{2}}.
\end{align}
Using the shorthand $s_\ell = e_{\ell}\transp \matbs^{-1} e_{\ell}$ and substracting $s_\ell \atom_\ell$ we get
\begin{align*}
  \|\bar \patom_\ell - s_\ell \atom_\ell\| \leq |s_\ell| \cdot \| \dico \mata (D_{\pi \cdot \alpha \cdot \coefficient})^{-1}  e_{\ell} - \atom_\ell \| + \| \dico \mata (D_{\sqrtpi \cdot \alpha \cdot \coefficient})^{-1} \| \cdot \|\mathbb{I}_{\ell^c} \matbs^{-1} e_{\ell} \pi_{\ell}^{-\frac{1}{2}}\|.
\end{align*}
The first term is well-behaved and makes no problems. Indeed, since $|s_\ell| \leq \|T^{-1}\|$, using \eqref{T_inv_bound} and Lemma~\ref{lemma_claim3} yields
\begin{align}
   |s_\ell| \cdot \| \dico  \mata (D_{\pi \cdot \alpha \cdot \coefficient})^{-1} e_{\ell} - \atom_{\ell} \|\leq (1 -   \Delta/4)^{-1} \cdot   \Delta/8.\label{mod:ell2:2}
\end{align}
To bound $\| \mathbb{I}_{\ell^c} e_\ell  \matbs^{-1} e_{\ell} \|$ we use $\matbs^{-1}  = \mathbb{I}+ T^{-1}(\mathbb{I}-T)$ and $\mathbb{I}_{\ell^c} e_\ell=0$ to get
\begin{align*}
    \| \mathbb{I}_{\ell^c} \matbs^{-1} e_{\ell} \| =   \|  \mathbb{I}_{\ell^c} T^{-1}(\mathbb{I}-T) e_{\ell}\| 
    &=\|  \mathbb{I}_{\ell^c} T^{-1}(e_\ell e_\ell \transp + \mathbb{I}_{\ell^c})(\mathbb{I}-T) e_{\ell}\| \\
    &\leq \|  \mathbb{I}_{\ell^c} T^{-1}e_\ell \|\cdot \|\mathbb{I}-T\| +\|T^{-1}\| \cdot \|\mathbb{I}_{\ell^c}T e_{\ell}\|  .
\end{align*}
Rearranging the above and using \eqref{T_inv_bound} as well as Lemma~\ref{lemma_claim2} yields
\begin{align*}
\| \mathbb{I}_{\ell^c} \matbs^{-1} e_{\ell} \|  &\leq \frac{\|   \matbs^{-1} \|}{1-\| \mathbb{I}-\matbs\|}   \cdot \| \mathbb{I}_{\ell^c} \matbs e_{\ell}  \| \leq \frac{1 }{(1-  \Delta/4)^2} \cdot \| \mathbb{I}_{\ell^c} \matbs e_{\ell}  \|.
\end{align*}
Combining these observations with the bound from \eqref{eq:7} and  Lemma~\ref{lemma_claim4} yields
\begin{align}
  \|\bar \patom_\ell - s_\ell \atom_\ell\| \leq \frac{ 1 }{1-\Delta/4} \cdot \frac{\Delta}{8} +  \frac{ 1 }{(1-\Delta/4)^2} \cdot \frac{5}{4}  \cdot \underbrace{\frac{\alphamin \gamma \nu }{4 C} \cdot \| \mathbb{I}_{\ell^c} \matbs e_{\ell}\pi_\ell^{-\frac{1}{2}}\| }_{\leq \Delta /8\quad \eqref{lemma_claim4}} \leq \frac{ \Delta}{3} .
\end{align}

\paragraph{Normalisation} Combining the above results shows that with high probability, the dictionary update step of both algorithms before normalisation satisfies
\begin{equation}
    \| (\bar{\pdico} - \dico)\weights \|_{2,2} \leq  \Delta/4 \quad \text{and} \quad \max_\ell\|\bar\patom_\ell-  s_\ell\atom_{\ell}\| \leq \Delta/3,
\end{equation}
where $s_\ell=1$ in case of ODL and $|s_\ell - 1| \leq  \|T^{-1} -1\| \leq \Delta/3 $ in case of MOD.
So what is left to show is that the normalisation step at the end of each iteration does not interfere with convergence. Let $F : = \diag(\| \Bar{\patom}_i\|_2)^{-1}$ be the square diagonal normalization matrix and denote by $\Hat{\pdico} := \Bar{\pdico} F$ the normalized dictionary of the current update step. Since $\| \atom_\ell \| =1 $ we have
$$ |\|\bar\patom_\ell\| -1 | \leq  \|\bar\patom_\ell- \atom_\ell\| \leq \|\bar\patom_\ell-s_\ell  \atom_\ell\| + \|(s_\ell-1)\atom_\ell \| \leq  \Delta/3 +  \Delta/3 \leq \Delta,$$
which further means that $ \| F\|_{2,2} \leq (1- \Delta)^{-1} $ and $ \|\mathbbm{I} - F \| \leq \Delta\cdot (1- \Delta)^{-1}$.
Hence, using that by assumption $ \|\dico \weights\| \leq 1/(4C)$ and $ \Delta \leq 2/C$ the weighted operator norm of the difference of the generating dictionary $\dico$ and the normalised update $\hat{\pdico}$ can be bounded as
\begin{align}
    \| (\hat{\pdico} - \dico) \weights \| & \leq  \|  (\Bar{\pdico} - \dico)\weights  \|\|F\| + \|\dico \weights\| \| (\mathbbm{I}-F)\| \nonumber \\
    & \leq \frac{ \Delta}{4} \cdot \frac{1} {1- \Delta} + \frac{1}{4C} \cdot \frac{\Delta}{1-\Delta} \leq \frac{ \Delta}{2}.
\end{align} 
To bound the $\ell_2$-norm we simply use Lemma B.10 from \cite{sc15}, which says that if $\|\atom\|=1$ and $\| \patom - s \atom\| \leq t$, then $\hat \patom = \patom/ \|\patom\|$ satisfies $\| \hat \patom - \atom \|^2 \leq 2 - 2 \sqrt{1 - t^2/s^2}$. 
Combining this with the bound $\sqrt{1-t}\geq 1 - \tfrac{t}{2-t}$ for $t \in (0,1)$ and using $t=\Delta/3$ and $|1-s_\ell| \leq \Delta/3$, meaning $s_\ell\geq 1 -\Delta/3$, yields
\begin{align*}
\|\hat \patom_\ell - \atom_\ell \| \leq t \cdot \left(s_\ell^2 - \frac{t^2}{2}\right)^{-1/2} \leq \frac{ \Delta}{3}\cdot \left(\Big(1-\frac{\Delta}{3}\Big)^2 - \frac{\Delta^2}{18}\right)^{-1/2}\leq \frac{\Delta}{2}.
\end{align*}
This shows that not only $\| (\hat{\pdico} - \dico) \weights \|$ but also $\|\hat{\pdico} - \dico \|_{2,1}$ and thus $\delta(\hat{\pdico},\dico)$ is bounded by $\Delta/2$ and thus concludes the proof of Theorem~\ref{the_theorem}. 
\end{proof} 


\noindent
\section{Discussion}\label{sec:discussion}

In this paper we have shown that two widely used alternating minimisation algorithms, MOD and ODL (aK-SVD) both combined with thresholding, converge to a well behaved data-generating dictionary $\dico$ from any initialisation $\pdico$ that either lies within distance $O(1/\log K)$ or is itself well-behaved and has a diagonally dominant
cross-coherence matrix $\pdico\transp \dico$. For ODL this constitutes the first convergence result of this kind, while for MOD it extends the convergent areas, derived in \cite{aganjane16}, by orders of magnitude and under weaker assumptions on the generating dictionary.\\
We want to emphasize that --- to the best of our knowledge --- our convergence theorem is the first result in theoretical dictionary learning, which is valid for signal generating models that do not assume a (quasi)-uniform distribution of sparse supports. Instead it can handle more realistic distributions, where each atom is used with a different probability. This stability might also be an explanation for the popularity of alternating minimisation algorithms in practical applications such as image processing. Natural images, for instance, contain more low than high-frequencies and the success of e.g. wavelets in classical signal processing can be attributed to the fact that they can be divided into frequency bands with different appearance probabilities. Since also more recent schemes such as compressed sensing can by improved by exploiting information about the inclusion/appearance probabilities of each wavelet, \cite{rusc22,ru22}, we expect that our study of dictionary learning for non-uniform support distributions will encourage the development of exciting new signal processing methods that leverage this additional information about the frequency of appearance (and thus in some sense the importance) of different atoms.\\
Note that the techniques we used to prove convergence for MOD and ODL here can also be used to turn the contraction results for ITKrM, \cite{pasc23}, into convergence results, \cite{ru23}. As part of our future work we plan to further increase the realism of our data models by including different coefficient distributions for different atoms and noise. We also plan to analyse partial convergence of the dictionary, meaning convergence for most but not all atoms. The main motivation for this is that a randomly initialised dictionary is still rather unlikely to satisfy even the relaxed diagonal dominance requirement. However, we expect that after resorting, a large, left upper part of $\pdico\transp \dico$ will be diagonally dominant and that all associated atoms still converge quite closely. A result of this kind would put the replacement and adaptive dictionary learning strategies, developed in \cite{pasc23}, on a firmer theoretical basis. Lastly, we are interested in increasing the realism of our signal model even further by including also second order inclusion probabilities into our generating model. This is again inspired by the behaviour of wavelets in natural images, where it is know that spatially close wavelets are likely to appear together.

\acks{This work was supported by the Austrian Science Fund (FWF) under Grant no.~Y760.}

\newpage
\appendix



\section{Proofs of the 4 inequalities}\label{app:4_inequalities}
\noindent
A major hurdle in analysing the above dictionary learning algorithms is that each update of the sparse coefficients involves projecting onto submatrices of the current guess $\pdico$. 
For the remainder of this chapter we will set $\vartheta:= 1/4$ and write 
\begin{align}
\mathcal{F}_\dico := \left\{ I \; : \; \| \dico_{I} \transp \dico_{I} - \mathbbm{I} \| \leq \vartheta \right\}\quad \mbox{and} \quad \mathcal{F}_\pdico: = \left\{ I \; : \; \|  \pdico_{I} \transp \pdico_{I} - \mathbbm{I} \| \leq \vartheta \right\}
\end{align}
for the set of index sets where the random variables $\dico_I$ resp. $\pdico_{I}$ are well conditioned. We further write 
\begin{align}
\mathcal{F}_Z := \left\{ I \; : \; \| Z_{I} \| \leq \delta \cdot \sqrt{2\log(n K \rho / \deltamin)} \right\}
\end{align} for the set of index sets, where the norm of the random variable $Z_I$ is comparable to $\delta$.
Finally, set 
\begin{align}
\mathcal{G} := \mathcal{F}_\dico \cup \mathcal{F}_\pdico \cup \mathcal{F}_Z.\label{defG}
\end{align}
We also need to control the sparse approximation step in each iteration. Recall that thresholding amounts to finding the largest $S$ entries in magnitude of $\pdico\transp y$, collecting them in the index set $\hat I$ and calculating the corresponding optimal coefficients as $\hat{x}_{\hat{I}} = \pdico_{\hat{I}}^\dagger y$. For the remainder of this chapter, we write
\begin{align}
\mathcal{H} := \left\{ (I, \sigma, c) \; | \; \hat{I} = I \right\} 
\end{align}
for the set of index, sign and coefficient triplets, where thresholding is guaranteed to recover the correct support.
With all the necessary notation in place, we can show that under the assumptions of Theorem~\ref{the_theorem}, the failure probability of thresholding and the probability that our submatrices are ill-conditioned can be bound by approximately $\deltamin/\rho$. This will be used repeatedly by the lemmas afterwards.
\begin{lemma}\label{prob_bounds}
Under the assumptions of Theorem~\ref{the_theorem} we have
\begin{align}
    \left[2\P(\mathcal{H}^c) + \P(\mathcal{G}^c) \right] \cdot \rho \leq \deltamin/32.
\end{align}
\end{lemma}
\begin{proof}
We begin with bounding the failure probability of thresholding, $\P(\mathcal{H}^c)$. Set $\holb: = \pdico\transp \dico - \diag(\pdico\transp \dico)$. By definition of the algorithm, thresholding recovers the full support of a signal $y = \dico_I x_I$, if
\[
\| \pdico_{I^c}\transp y \|_{\infty} <  \| \pdico_I\transp y\|_{\min}.
\]
Note that the signals have two sources of randomness, $\sigma$ and $I$. Recall that $\alphamin = \min_{i} \absip{\patom_i}{\atom_i}$. Plugging in the definition of $y$, we bound the failure probability as
\begin{align}
\P_y(\| \pdico_I\transp y\|_{\min} & <  \| \pdico_{I^c}\transp y \|_{\infty})  =  \P_y \left( \| \pdico_I\transp \dico_I x_I \|_{\min} <  \| \pdico_{I^c}\transp \dico_I x_I \|_{\infty} \right) \nonumber \\
&\leq  \P_y \left( c_{\min} \| \diag(\pdico_I\transp \dico_I)\|_{\min}  - \| \holb_{I,I} x_I \|_{\infty}  <  \| \pdico_{I^c}\transp \dico_I x_I \|_{\infty}  \right) \nonumber \\
& \leq \P_y \left( c_{\min} \cdot \alphamin < 2 \| \holb_I x_I \|_{\infty}  \right) \nonumber \\
& \leq \P_y\left(2\| \holb_I x_I \| _{\infty}  \geq c_{\min}\cdot  \alphamin \; \big| \;  \|\holb_I\|_{\infty,2}  < \eta \! \right) + \P_S \left(\|\holb_I\|_{\infty,2}  \geq \eta  \right).\label{eq:08}
\end{align}
To bound the first term, we use that for $k \in I$, we have $x_{k} = \sigma_k c_k $, where $\sigma \in \R^{S}$ is an independent Rademacher sequence. As the signs $\sigma$ are independent of the support $I$, we can apply Hoeffding's inequality to each entry of $\holb_I x_I $. The second term is a little more involved. By the Poissonisation trick~{\cite[Lemma~3.5]{rusc21}} we have
\[
\P \left( \|\holb_I\|_{\infty,2}  \geq \eta  \right) \leq 2 \P_B \left( \|\holb_I\|_{\infty,2}  \geq \eta  \right) ,
\]
where $\P_B$ is the Poisson sampling model corresponding to the $p_1, \cdots , p_K$. Now a simple application of the matrix Chernoff inequality, \cite{tr12}, restated in \ref{chern}, together with an application of Hoeffding's inequality to the first term in~\eqref{eq:08} yields
\begin{align}
\P_y(\| \pdico_I\transp y\|_{\min} <  \| \pdico_{I^c}\transp y \|_{\infty}) 
&\leq 
2K\exp{\left(- \frac{c_{\min}^2}{8 c_{\max}^2 \eta^2}\cdot \alphamin^2 \right) } + 2K \left( e \frac{\| \holb D_{\sqrt{p}} \|_{\infty,2}^2}{\eta^2}  \right) ^{\frac{\eta^2}{\mu(\pdico,\dico)^2}} \nonumber \\ 
&\leq 
2K\exp{\left(- \frac{\gamma^2}{8 \eta^2} \cdot \alphamin^2 \right) } + 2 K \left( 2 e \frac{\| \holb D_{\sqrt{\pi}} \|_{\infty,2}^2}{\eta^2}  \right)^{\frac{\eta^2}{\mu(\pdico,\dico)^2}}, \label{thresh_fail_bound}
\end{align}
where we used that $(1- \|p\|_\infty)\cdot p_i \leq \pi_i$, from \cite{rusc22} resp. Theorem~\ref{th_poisson}(a), which implies $p_i \leq \tfrac{6}{5} \pi_i  <2\pi_i$ as well as $\| \holb D_{\sqrt{p}} \|_{\infty,2}^2 \leq 2\| \holb D_{\sqrt{\pi}} \|_{\infty,2}^2$. \\
Before going on we recall the abbreviations
$$
\nu = \frac{1}{\sqrt{\log(n K\rho /\deltamin)}}  \leq \frac{1}{3} \qquad \text{and}\qquad \delta_\circ = \frac{\gamma}{C\log(n K\rho /\deltamin)}=\frac{\gamma \nu^2}{C},
$$
where the bound on $\nu$ holds true since $n\rho\deltamin^{-1}\geq 130 \cdot 2 \cdot42$. Setting $\eta :=  \gamma \alphamin \nu/4$ the first term in \ref{thresh_fail_bound} becomes $2K \deltamin^2/(n \rho K)^2$. To bound the second term observe that  
\begin{align}
    \| \holb \weights \|_{\infty,2}^2 = \| [\pdico \transp \dico - \diag(\pdico \transp \dico)] \weights \|^2_{\infty,2}  \leq  \|  \dico \weights \|^2, \label{relax_dico_cond} 
\end{align}
so our condition on the generating dictionary in \eqref{cond_dico}, $\|\dico \weights \|\leq \gamma \alphamin \nu /(4C) $, ensures that $2 e \| \holb D_{\sqrt{\pi}} \|_{\infty,2}^2 \eta^{-2} \leq e^{-2}$. To lower bound the exponent of the second term we will have to look at both regimes described in Theorem~\ref{the_theorem}, separately. In the first regime, $\delta > \delta_{\circ}$, by \eqref{regime2} we simply have $\mu(\pdico,\dico)\leq  \alphamin\gamma \nu^2/(4C)$. In the second regime, $\delta \leq \delta_{\circ}$, we can employ the following bound 
$$
\mu(\pdico,\dico) = \max_{i\neq j} \absip{\atom_i}{\patom_j} \leq \max_{i\neq j} \absip{\atom_i}{\atom_j} + \max_j \|\patom_j - \atom_j\|\leq \mu(\dico) + \delta,\\
$$
together with the observation that whenever $\delta \leq \delta_{\circ}\leq 1/C$, we have $\alphamin = 1-\eps^2/2 \geq 1-\delta^2/2 \geq (C-1)/C$ and therefore also 
\begin{align}\label{regime2_alpha}
\delta \leq \frac{\alphamin\gamma}{\alphamin C \log(n K\rho /\deltamin)} \leq \frac{\alphamin\gamma}{(C-1) \log(n K\rho /\deltamin)}= \frac{\alphamin\gamma \nu^2}{(C-1)}.
\end{align}
This means that in both regimes we have $\mu(\pdico,\dico)\leq 2 \alphamin\gamma \nu^2/(C-1)$. Substituting these bounds into \ref{thresh_fail_bound}, we get 
\begin{align}
    \P \left( \mathcal{H}^c \right) \leq 2K \left(\frac{\deltamin}{n\rho K}\right)^2 + 2K \left(\frac{\deltamin}{n\rho K}\right)^{(C-1)^2/32} \leq 4 K \left(\frac{\deltamin}{n\rho K}\right)^2.\label{Hbound}
\end{align}
Now we turn to bounding the quantity $\P(\mathcal{F}_Z^c)$. Again by the Poissonisation trick, the matrix Chernoff inequality and the bound $p_i \leq 2\pi_i$ we have
\begin{align*}
    \P \big( \| Z_I Z_I \transp \| > t &\big) \leq 2 \P_B \big( \| Z_I Z_I \transp \| >t  \big) \nonumber \\
    &\leq 2 K \left(\frac{e\| Z D_{p} Z\transp \|}{t } \right) ^{t / \varepsilon^2 } \leq 2 K \left( \frac{2e\| Z D_{\pi} Z\transp \|}{t } \right) ^{t / \varepsilon^2 } \leq 2 K \left( \frac{2e\delta^2}{t } \right) ^{t / \delta^2 }.
\end{align*}
Setting $t= 2\delta^2 \max\{e^2, \log(n K \rho / \deltamin)\} = 2\delta^2 \log(n K \rho / \deltamin)$  we get
\begin{align}
    \P(\mathcal{F}_Z^c) =  \P \left( \| Z_{I}  \| \geq \delta \cdot \sqrt{2 \log(n K \rho / \deltamin)} \right)  \leq 2K \left(\frac{\deltamin}{nK \rho}\right)^{2} .\label{FZbound}
\end{align}
Next we use Theorem~\ref{them:opnorm}, \cite{rusc21}, and $p_i \leq \tfrac{6}{5} \pi_i$ to bound $\P(\mathcal{F}_\dico^c\cup \mathcal{F}_\pdico^c)$ as
\begin{align*}
    \P(\mathcal{F}_\dico^c \cup \mathcal{F}_\pdico^c) \leq 512 K \exp{\left(- \min\left\{ \frac{5\vartheta^2}{24 e^2 \|  \dico  \weights\|^2 } , \frac{\vartheta}{ 2 \mu(\dico)},\frac{5\vartheta^2}{24 e^2 \|  \pdico  \weights\|^2 } , \frac{\vartheta}{ 2 \mu(\pdico) } \right\}\right)}.
\end{align*}
In the first regime we have by assumption that $\max \{\nu \|  \dico \weights \| , \mu(\dico)\} \leq   \gamma \alphamin \nu^2/(4C)$ and $\max\{\nu  \| \pdico \weights \| , \mu(\pdico) \} \leq \gamma\alphamin \nu^2/(4C)$,
while in the second regime we only have the first inequality as well as $\delta\leq  \gamma \alphamin \nu^2/(C-1)$. However, we can use the bound
\begin{align}
\| \pdico \weights \| \leq \| \dico \weights \| + \| (\pdico - \dico) \weights \| &\leq  \| \dico \weights \| + \delta
\nonumber \\
& \leq \gamma \alphamin \nu \cdot ( \tfrac{1}{4C}+ \tfrac{\nu}{C-1})\leq \gamma \alphamin \nu/C, \label{bound_pdico_r2}
\end{align}
and the fact that for $\delta\leq \delta_\circ = \gamma \nu^2/C\leq \nu^2/C$ we have
\begin{align}
\mu(\pdico) = \max_{i \neq j} \absip{\patom_i}{\patom_j} &\leq  \max_{i \neq j } \vert \ip{\atom_i}{\atom_j} +  \ip{\atom_i }{\patom_j-\atom_j} +  \ip{\patom_i - \atom_i}{\patom_j} \vert  \nonumber \\
&\leq \mu(\dico) + 2\varepsilon  \leq \mu(\dico) + 2\delta \leq \tfrac{9}{4} \cdot \nu^2/C.
\end{align}
All together this means that for $\vartheta = 1/4$ we can lower bound the expression in the minimum in both regimes by $2 \log(n K \rho / \deltamin)$,
leading to $\P(\mathcal{F}_\dico^c\cup\mathcal{F}_\pdico^c) \leq 512 K \left(\deltamin/(n\rho K)\right)^2.$ Collecting all the bounds we finally get for $n\geq 130 $ that
\begin{align*}
    \P(\mathcal{H}^c) \cdot 2 \rho  + \P(\mathcal{G}^c) \cdot \rho \leq 522 K\rho \left(\frac{\deltamin}{n\rho K}\right)^2 \leq \frac{\deltamin}{32}. 
\end{align*}

\end{proof}
Another ingredient we need to prove the 4 inequalities used in the proof of our main theorem is the
following lemma to estimate expectations of products of random matrices. A detailed proof, based on \cite{sing_values_1, sing_values_2}, can be found in Appendix~\ref{sec:supp_rand_mat} .
\begin{lemma}\label{bmb}
Let $A(I) \in \R^{d_1 \times d_2}$, $B(I) \in \R^{d_2 \times d_3}$, $C(I)\in \R^{d_3 \times d_4}$ be random matrices, where $I$ is a discrete random variable taking values in $\mathcal{I}$ and $\mathcal{G}\subseteq \mathcal{I}$. If for all $I \in \mathcal{G}$ we have $\|B(I)\| \leq \Gamma$ then 
\[
\|\E \left[ A(I) \cdot B(I) \cdot C(I) \cdot \indi_{\mathcal G}(I) \right] \| \leq \| \E \left[ A(I) A(I)\transp \right] \|^{1/2} \cdot \Gamma \cdot \| \E \left[ C(I) \transp C(I) \right] \|^{1/2}.
\]
\end{lemma}
Finally, we also need the following corollary of results from \cite{rusc22}, which is again proved in Appendix~\ref{sec:supp_proof_corollary}.
\begin{corollary} \label{the_corollary}
Denote by $\E$ the expectation according to the rejective sampling probability with level $S$ and by $\pi \in \R^K$ the first order inclusion probabilities of level $S$. Let $\hol$ be a $K\times K$ matrix with zero diagonal, $W=(w_1 \ldots ,w_K)$ and $V=(v_1, \ldots ,v_K)$ a pair of ${d\times K}$ matrices and $\mathcal{G}$ a subset of all supports of size $S$, meaning $\mathcal{G} \subseteq \{ I : |I| = S\}$. If $\|\pi\|_\infty \leq 1/3$, we have 
\begin{align}
    \| \E[ \weights^{-1}  R_I\transp \hol_{I,I} R_I  \weights ^{-1}]\| & \leq 3\cdot \| \weights \hol \weights \|,  \tag{a}\label{E_RtNR} \\
    \| \E[\weights^{-1}  R_I\transp \hol_{I,I} \hol_{I,I} \transp R_I  \weights^{-1}]\|  & \leq \tfrac{9}{2}  \cdot\| \weights \hol \weights\|^2 + \tfrac{3}{2} \cdot \max_k \|e_k\transp \hol \weights\|^2, \tag{b}\label{E_RtNNtR_gen}\\
    \| \E[  W R_I\transp R_I V\transp \cdot \indi_I(\ell) \indi_{\mathcal{G}}(I)]\|&\leq \pi_\ell \cdot \left(\|W\weights\|\cdot \| V \weights \| +\|w_\ell\|\cdot \|v_\ell\|\right), \tag{c}\label{E_WRtRVt}
    \end{align}
    as well as
\begin{align}
    \| \E [ &\weights ^{-1} \mathbb{I}_{\ell^c} R_I \transp \hol_{I,I} \hol_{I,I}\transp R_I\mathbb{I}_{\ell^c} \weights ^{-1}  \cdot \indi_I(\ell)]\| \notag \\
     & \leq \tfrac{3}{2} \cdot \pi_\ell \cdot \left(3\cdot \| \weights \hol e_\ell \|^2 + {\max}_{k} \hol_{k\ell}^2 + \tfrac{9}{2} \cdot \| \weights \hol \weights\|^2 + \tfrac{3}{2} \cdot {\max}_k \|e_k\transp \hol\weights\|^2 \right). \tag{d} \label{E_RtNNtR_indl_gen}  
 \end{align}
\end{corollary}
With the last three results in place we are finally able to prove Lemmas~\ref{lemma_claim1}-\ref{lemma_claim4}, which provide the 4 inequalities our proof is based on.

\begin{lemma}\label{lemma_claim1}
Under the assumptions of Theorem~\ref{the_theorem} we have 
\begin{align}
\P \bigg( \| \dico \mata (D_{\sqrtpi  \cdot \alpha \cdot \coefficient })^{-1} - \dico \weights \| >  \alphamin \Delta/8\bigg) \leq  (d+K) \exp{\left(- \frac{N (  \Delta/16)^2 }{2 \rho^2+ \rho  \Delta/16 }\right)}.\label{bound_claim1}
\end{align} 
\end{lemma}
\begin{proof}
The idea is to write $\dico \mata (D_{\sqrtpi  \cdot \alpha \cdot \coefficient})^{-1} - \dico \weights $ as a sum of independent random matrices and apply the matrix Bernstein inequality to show that we have concentration. Since we assumed in the algorithm that the estimated coefficients can never have larger norm than the signal times $\kappa$ we first define for $v\in \R^d$ the set of possible stable supports as $\mathcal{B}(v): = \{ I : \|\pdico_I^\dagger v\| \leq \kappa \|v\| \}$. Based on this definition we define the following random matrices for $n \in [N]$ 
\begin{align*}
     \hat{Y}_n :&=  y_n y_n \transp  (\pdico_{\hat{I}_n}^{\dagger})\transp R_{\hat{I}_n}  (D_{\sqrtpi  \cdot \alpha \cdot \coefficient})^{-1}\cdot \indi_{\mathcal{B}(y_n)}(\hat{I}_n) -  \dico  \weights, \nonumber 
\end{align*}
where as always, $\hat{I}_n$ denotes the set found by the thresholding algorithm. As each matrix $\hat{Y}_n$ only depends on the signal $y_n$ they are independent and we have 
$$N^{-1} \textstyle \sum_n \hat Y_n = \dico A (D_{\sqrtpi  \cdot \alpha \cdot \coefficient})^{-1} - \dico \weights,$$
so we can use the matrix Bernstein inequality~\ref{bernstein} from \cite{tr12} to bound the left hand side of \eqref{bound_claim1}. For that we have to find an upper bound for the operator norm. By the assumptions on the generating dictionary and since we ensured in the algorithm that the estimated coefficients, which are too large are set to zero, meaning $\|\hat{x}_n \| \leq \kappa \| y_n \| $ and $\|y_n \| \leq S c_{\max} $, we get for $\rho =  2 \kappa^2 S^2 \gamma^{-2} \alphamin^{-2}\pimin^{-3/2}$ and $\pimin < 1/3$,
\begin{align}
\| \hat{Y}_n \|  &\leq \kappa S^2 c_{\max}^2 \|\coeff^{-1}\| \|\coeffm^{-1}\| \|\weights^{-1}\|  + \| \dico \weights\|  \leq  \rho \alphamin \pimin + 1 \leq  \rho \alphamin/2 =: r.
\end{align}
Bounding $\|\E[\hat{Y}_n]\|$ is a little more involved. Recall that $\mathcal{H}$ is the set of signals~$y$, meaning support, sign and coefficient triplets $(I,\sigma,c)$, where thresholding recovers the correct support from the corresponding signal. Further $\mathcal{G}$ is the set of supports $I$ where $\vartheta_I$ is small - i.e. the corresponding subdictionary $\pdico_I$ is well-conditioned. For each $n$ we define a new random matrix $Y_n$, for which the estimated support $\hat{I}_n$ is replaced with the correct support $I_n$ and $\dico \weights$ is replaced by $\dico \diag(\one_{I_n})\weights^{-1}$ 
\begin{align*}
   Y_n : &=y_n y_n \transp  (\pdico_{I_n}^{\dagger})\transp R_{I_n}  (D_{\sqrtpi  \cdot \alpha \cdot \coefficient})^{-1}\cdot \indi_{\mathcal{B}(y_n)}(I_n) -  \dico \diag(\one_{I_n})\weights^{-1} .\nonumber
\end{align*}
Note that $\| \dico \diag(\one_{I_n})\weights^{-1} \| \leq S\pimin^{-1/2}$ so the same bound as for $\|\hat Y_n\|$ holds. Concretely, with the same argument as above $Y_n$ is bounded by $r$. Further, by definition of $\mathcal H$ the first terms of the two random matrices $Y_n$ and $\hat{Y}_n$ coincide on $\mathcal H$, while the second terms coincide in expectation, meaning $\E[\dico \diag(\one_{I_n})\weights^{-1}] = \dico \weights$. So dropping the index $n$ for convenience, as each signal has the same distribution, e.g., writing $I$ for $I_n$, we get using Lemma~\ref{prob_bounds},
\begin{align}
\| \E[ \hat{Y} ] \| \leq  \| \E[ \hat{Y} - Y]\|  +\|\E[Y] \|& \leq  \P(\mathcal{H}^c) \cdot 2 r +\|\E[\indi_{\mathcal{G}^c}(I)Y]\| + \|\E[\indi_{\mathcal{G}}(I)Y]\| \nonumber \\
&\leq  \P(\mathcal{H}^c) \cdot 2 r + \P(\mathcal{G}^c) \cdot r + \|\E[\indi_{\mathcal{G}}(I)Y]\| . \label{c1_HGsplit}
\end{align}
Next note that whenever $I\in \mathcal G$ we have for any sign and coefficient pair $(\sigma,c)$ that the corresponding signal~$y$ satisfies $ \|\pdico_I^\dagger y\| \leq (1-\vartheta)^{-\frac{1}{2}} \cdot \|y\| \leq \kappa \|y\|$, so we have $\mathcal{G} \subseteq B(y)$, meaning $\indi_{\mathcal{B}(y)}\indi_{\mathcal{G}} = \indi_{\mathcal{G}}$. Remembering that $y = \dico_I x_I =\dico_I (\sigma_I \odot c_{I})$, we can take the expectation over $(\sigma,c)$, which, using the shorthand $\E_{\mathcal{G}}[ f(I)]:=\E_I[\indi_{\mathcal{G}}(I) f(I)]$, yields
   \begin{align}
  \|\E[\indi_{\mathcal{G}}(I)Y]\| 
   &=
     \| \E_I [\indi_{\mathcal{G}}(I) \cdot \E_{\sigma,c} [\dico_I x_I x_I \transp \dico_I \transp (\pdico_I^{\dagger})\transp R_I  (D_{\sqrtpi  \cdot \alpha \cdot \coefficient})^{-1}  -  \dico \diag(\one_I)  \weights^{-1}]] \| \nonumber \\
    & = \| \E_{\mathcal{G}}[\dico_I  \dico_I \transp (\pdico_I^{\dagger})\transp R_I  (D_{\sqrtpi  \cdot \alpha } )^{-1} -   \dico_I (D_\alpha)_{I,I} R_I (D_{\sqrtpi  \cdot \alpha } )^{-1}  ] \|. \label{c1_EG_signs_coeff}
    \end{align}
    We next have a closer look at the term $ \dico_I \transp (\pdico_I^{\dagger})$. Set $H = \mathbb{I} - \pdico\transp \pdico$. For $I \in \mathcal{G}$ we have $\| H_{I,I}\| = \|  \mathbbm{I}_{S} - \pdico_I \transp \pdico_I  \| \leq \vartheta$, meaning $\pdico_I$ has full rank and $\pdico^{\dagger}_I =(\pdico_I \transp \pdico_I)^{-1}\pdico_I\transp.$ We can further use the Neumann series to get the useful identity
\begin{align}
     (\pdico_I \transp \pdico_I)^{-1} =  ( \mathbbm{I}_{S} - H_{I,I})^{-1}   =   \textstyle \sum_{k \geq 0}H_{I,I}^k 
     &= \mathbbm{I}_{S} + (\pdico_I \transp \pdico_I)^{-1}H_{I,I}\label{neumann_Hright}\\
     & = \mathbbm{I}_{S} + H_{I,I}(\pdico_I \transp \pdico_I)^{-1} ,\label{neumann_Hleft} 
\end{align}
and the norm estimate $\|(\pdico_I \transp \pdico_I)^{-1}\| \leq (1-\vartheta)^{-1}$. By definition of $Z = \pdico - \dico$, we have $\holdiag: = \diag(Z \transp \pdico ) = \mathbbm{I} - \coeffm $. We also define the zero diagonal matrix 
\begin{align}
    \hol := (\pdico \holdiag -Z)\transp \pdico = (\dico \holdiag - Z D_\alpha )\transp \pdico = (\dico - \pdico D_\alpha)\transp \pdico, \label{hol_expressions}
\end{align} which lets us express $\dico_I \transp (\pdico_I^{\dagger})\transp $ as
\begin{align}
    \dico_I \transp (\pdico_I^{\dagger})\transp  & =   (\pdico_I \transp -Z_I\transp) (\pdico_I^{\dagger})\transp = \mathbb{I}_S - Z_I \transp (\pdico_I^{\dagger})\transp \notag \\
    &= (\coeffm)_{I,I} +\holdiag_{I,I}\pdico_I \transp \pdico_I (\pdico_I \transp \pdico_I)^{-1}  - Z_I\transp \pdico_I (\pdico_I \transp \pdico_I)^{-1} \notag\\
    &= (\coeffm)_{I,I} +\hol_{I,I} (\pdico_I \transp \pdico_I)^{-1}= (\coeffm)_{I,I} +\hol_{I,I} + \hol_{I,I}(\pdico_I \transp \pdico_I)^{-1}H_{I,I}.\label{dico_transp_pdico_dagger} 
\end{align}
Note that for any $I$ we have $\|\hol_{I,I} \| \leq \|(\pdico \holdiag -Z)_I\| \cdot \|\pdico_I\| \leq  \eps \sqrt{S} \cdot \sqrt{S} < 2S$ and therefore $\|\dico_I  \hol_{I,I} R_I  (D_{\sqrtpi  \cdot \alpha } )^{-1}\| \leq  \rho \alphamin/2 = r$. 
So substituting the expression for $\dico_I \transp (\pdico_I^{\dagger})\transp$ above into \eqref{c1_EG_signs_coeff} resp. \eqref{c1_HGsplit} and rewriting $\dico_I = \dico \weights \weights^{-1} R_I\transp$ we get
\begin{align}
\| \E[ \hat{Y} ] \| &\leq [2\P(\mathcal{H}^c)+ \P(\mathcal{G}^c)] \cdot r + \| \E_{\mathcal{G}}[\dico_I (\hol_{I,I}+\hol_{I,I}(\pdico_I \transp \pdico_I)^{-1}H_{I,I}) R_I (D_{\sqrtpi  \cdot \alpha } )^{-1}  ] \| \notag\\
&\leq [2\P(\mathcal{H}^c)+ \P(\mathcal{G}^c)] \cdot r + \P(\mathcal{G}^c) \cdot r  + \| \E[\dico_I \hol_{I,I} R_I (D_{\sqrtpi  \cdot \alpha } )^{-1}  ] \| \notag \\
& \hspace{5cm} +\| \E_{\mathcal{G}}[\dico_I \hol_{I,I}(\pdico_I \transp \pdico_I)^{-1}H_{I,I} R_I (D_{\sqrtpi  \cdot \alpha } )^{-1}  ] \|\notag\\
&\leq [\P(\mathcal{H}^c)+ \P(\mathcal{G}^c)] \cdot \rho \alphamin  + \|\dico \weights\| \cdot \| \E[\weights^{-1} R_I \transp\hol_{I,I} R_I \weights^{-1}]\| \cdot \| D_{ \alpha } ^{-1} \| \notag\\
& \hspace{2.3cm} + \|\dico \weights\| \cdot \| \E_{\mathcal{G}}[\weights^{-1} R_I \transp \hol_{I,I}(\pdico_I \transp \pdico_I)^{-1}H_{I,I} R_I \weights^{-1}]\| \cdot \| D_{ \alpha } ^{-1}\|. \label{bound_c1_2HG}
\end{align}
Using Corollary~\ref{the_corollary}\eqref{E_RtNR} and $\eps \leq \min\{\delta, \sqrt{2}\}$, we bound the first expectation as
\begin{align}
   \| \E[\weights^{-1} R_I \transp \hol_{I,I} R_I \weights^{-1}]\|  
 \leq  3 \| \weights \hol \weights  \|  &\leq 3 \cdot  ( \|\dico \weights\| \cdot \eps^2/2 + \| Z \weights\|) \cdot\| \pdico \weights\|\notag\\
& \leq 3 \cdot \delta \cdot \| \pdico \weights\| \cdot ( \|\dico \weights\| + 1). \label{bound_c1a}
\end{align}
Before we estimate the second expectation, note that applying Corollary~\ref{the_corollary}\eqref{E_RtNNtR_gen} to $H,\hol,\hol \transp$ and using that $\max\{\|\pdico\weights\|,\|\dico \weights\| \}\leq \nu/C \leq 1/8$ yields the following three bounds, whose derivation can be found in Appendix~\ref{sec:supp_proof_corollary},
\begin{align}
    \| \E[ \weights^{-1}  R_I\transp H_{I,I} H_{I,I} \transp R_I  \weights^{-1}]\| &\leq 2 \cdot \| \pdico \weights \|^2 ,\label{E_RtHHR}\\
    \| \E[ \weights^{-1}  R_I\transp \hol_{I,I} \hol_{I,I} \transp R_I  \weights^{-1}]\|
    &  \leq 1/2 \cdot (3\|Z\weights\| + 3\eps )^2 \cdot \|\pdico\weights\|^2, \label{E_RtNNtR}\\
    \| \E[ \weights^{-1}  R_I\transp \hol_{I,I} \transp \hol_{I,I} R_I  \weights^{-1}]\| 
    &\leq 2 \cdot (\|\dico \weights\| \cdot \eps^2/2 + \|Z \weights\|)^2  \label{E_RtNtNR} .
\end{align}
Applying Theorem~\ref{bmb} to the second expectation in \eqref{bound_c1_2HG}, using that on $\mathcal{G}$ we have $\| (\pdico_I \transp \pdico_I)^{-1}\|\leq (1-\vartheta)^{-1} \leq 4/3$, and the first two inequalities above yields
\begin{align}
   \| \E_{\mathcal{G}}&[\weights^{-1} R_I\transp \hol_{I,I}\cdot (\pdico_I \transp \pdico_I)^{-1} \cdot H_{I,I} R_I \weights^{-1}]\|\notag \\
   &\leq  \| \E[\weights^{-1} R_I\transp \hol_{I,I}\hol_{I,I}\transp R_I\transp \weights \|^{1/2} \cdot 4/3 \cdot\| \E[\weights^{-1} R_I  H_{I,I}\transp H_{I,I}R_I  \weights^{-1} ] \|^{1/2} \notag\\
   &\leq \big(4\|Z\weights\| + 4\eps \big)\cdot \|\pdico\weights\|^2 \leq 8 \cdot \delta \cdot \|\pdico\weights\|^2. \label{bound_c1b}
   \end{align}
Substituting \eqref{bound_c1a}, \eqref{bound_c1b} and the probability bound from Lemma~\ref{prob_bounds} into \eqref{bound_c1_2HG} leads to
   \begin{align}
\| \E[ \hat{Y} ] \|\leq \alphamin \deltamin /32 + \delta/\alphamin \cdot\|\dico \weights\| \cdot  \| \pdico \weights\| \cdot (3 + 3\|\dico \weights\| + 8\|\pdico \weights\|). \label{c1_final_bound}
\end{align}
By the assumptions of Theorem~\ref{the_theorem} we have in both regimes $\|\dico \weights\| \leq  \gamma \alphamin \nu/(4C)$ and $ \|\pdico \weights\| \cdot \delta \leq  \gamma \alphamin \nu/(2C)$, while due to \eqref{bound_pdico_r2} we have $\|\pdico \weights\| \leq \gamma \alphamin\nu/C $ again in both regimes. Bounding the quantity in \eqref{c1_final_bound} in two ways we get
\begin{align*}
    \| \E[ \hat{Y} ] \|  
    & \leq \frac{\alphamin\deltamin }{32}  + \frac{\alphamin \gamma}{C} \cdot  \min\left\{ \frac{ \gamma \nu^2}{C}, \frac{2 \gamma \nu^2}{C} \cdot \delta\right\}  \leq \frac{\alphamin\deltamin }{32} + \frac{\alphamin \gamma}{C} \cdot \min \{\delta_\circ, \delta \} \leq   \alphamin \cdot \frac{ \Delta}{16}.
\end{align*}
Finally, an application of the matrix Bernstein inequality~\ref{bernstein} with $t=m=\alphamin  \Delta /16$ and $r =  \alphamin \rho/2$ and some simplifications yield the desired bound.
\end{proof}
The next lemma shows that the matrix $B = \sum_{n = 1}^N \hat{x}_n \hat{x}_n\transp$ essentially behaves like a diagonal matrix. In particular, after rescaling we have
$
(D_{\sqrtpi  \cdot \alpha } )^{-1} \matb (D_{\sqrtpi  \cdot \alpha \cdot \coefficient})^{-1} \approx \mathbbm{I}.
$

\begin{lemma}\label{lemma_claim2}
Under the assumptions of Theorem~\ref{the_theorem} we have
\begin{align*}
    \P \left( \| (D_{\sqrtpi  \cdot \alpha } )^{-1} \matb (D_{\sqrtpi  \cdot \alpha \cdot \coefficient}   )^{-1} - \mathbbm{I}\| >  \Delta/4 \right)
    \leq 2K \exp{\left(- \frac{N ( \Delta/16 )^2 }{2 \rho^2+ \rho  \Delta/16 }\right)}.\label{eq:lemma_claim2}
\end{align*}
\end{lemma}

\begin{proof}
We will follow the approach in last proof very closely, that is, we write the matrix $(D_{\sqrtpi  \cdot \alpha } )^{-1} \matb (D_{\sqrtpi  \cdot \alpha \cdot \coefficient})^{-1} - \mathbbm{I}$ as a scaled sum of independent random matrices $\hat Y_n$ and apply the matrix Bernstein inequality. Recalling that $\hat{I}_n$ denotes the set found by thresholding and that $\mathcal{B}(v): = \{ I : \|\pdico_I^\dagger v\| \leq \kappa \|v\| \}$ denotes the set of possible stable supports for $v$, we define for $n \in [N]$ the matrices $\hat Y_n$ as well as their auxiliary counterparts $Y_n$ as
\begin{align*}
     \hat{Y}_n :&=  (D_{\sqrtpi  \cdot \alpha }  )^{-1}R_{\hat{I}_n}\transp \pdico_{\hat{I}_n}^{\dagger} y_n y_n \transp (\pdico_{\hat{I}_n}^{\dagger })\transp R_{\hat{I}_n} (D_{\sqrtpi  \cdot \alpha \cdot \coefficient} )^{-1} \indi_{\mathcal{B}(y_n)}(\hat{I}_n) - \mathbbm{I} \nonumber \\
\mbox{and} \quad     Y_n : &= (D_{\sqrtpi  \cdot \alpha }  )^{-1} R_{I_n}\transp \pdico_{I_n}^{\dagger} y_n y_n \transp (\pdico_{\hat{I}_n}^{\dagger })\transp R_{I_n} (D_{\sqrtpi  \cdot \alpha \cdot \coefficient} )^{-1} \indi_{\mathcal{B}(y_n)}(I_n) - \diag(\one_{I_n})\weightss^{-1}.\nonumber 
\end{align*}
Recall that $\rho =  2 \kappa^2 S^2 \gamma^{-2} \alphamin^{-2}\pimin^{-3/2}$, so both matrices can be bounded as
\begin{align}
\max\{\| \hat{Y}_n\|, \| Y_n\|\} &\leq \kappa^2 S^2 c_{\max}^2 \|\coeff^{-1}\|\|\coeffm^{-2}\| \| \weightss^{-1}\| + \| \weightss^{-1}\| \leq 3 \rho/4 =: r.
\end{align}
On $\mathcal H$, meaning whenever thresholding succeeds, the first terms of $\hat Y_n$ and $Y_n$ again coincide while the second terms are the same in expectation, that is $\E[\diag(\one_{I_n})\weightss^{-1}] = \mathbbm{I}$. So with the same argument as in \eqref{c1_HGsplit} and as usual dropping the index $n$ for convenience, we get
\begin{align}
   \| \E[ \hat{Y} ] \| 
   &\leq  2\rho \cdot \P(\mathcal{H}^c) + \rho \cdot \P(\mathcal{G}^c) + \|\E[\indi_{\mathcal{G}}(I) Y] \| .\label{c2_EY}
   \end{align}
   Similarly as in \eqref{c1_EG_signs_coeff} we next use that all well conditioned supports are stable for any signal~$y$, meaning $\mathcal{G} \subseteq \mathcal B(y)$. Taking the expectation over $(\sigma,c)$ yields
   \begin{align}
   \|\E[\indi_{\mathcal{G}}(I)Y] \| 
   &=
     \| \E_{\mathcal{G}}  \E_{\sigma,c} [(D_{\sqrtpi  \cdot \alpha }  )^{-1} R_I \transp \pdico_I^{\dagger} \dico_I x_Ix_I\transp \dico_I \transp (\pdico_I^{\dagger})\transp R_I (D_{\sqrtpi  \cdot \alpha \cdot \coefficient}  )^{-1}  \!  - \!  \diag(\one_I) \weightss^{-1}] \| \nonumber \\
     &= \big\| \E_{\mathcal{G}} \big[(D_{\sqrtpi  \cdot \alpha }  )^{-1} R_I \transp \big(\pdico_I^{\dagger} \dico_I \transp \dico_I \transp (\pdico_I^{\dagger})\transp-(\coeffm)^2_{I,I} \big)R_I (D_{\sqrtpi  \cdot \alpha }  )^{-1}\big] \big\|. \label{c2_EGYa}
\end{align}
Using the expression for $\dico_I \transp (\pdico_I^{\dagger})\transp$ from \eqref{dico_transp_pdico_dagger} we get
\begin{align}
     \pdico_I^{\dagger} \dico_I   \dico_I \transp (\pdico_I^{\dagger})\transp - (\coeffm^2)_{I,I}  
     =   (\coeffm)_{I,I} \hol_{I,I}(\pdico_I \transp \pdico_I)^{-1} &+ (\pdico_I \transp \pdico_I)^{-1}\hol_{I,I} \transp(\coeffm)_{I,I}\notag \\
     &  +(\pdico_I \transp \pdico_I)^{-1}\hol_{I,I}\transp \hol_{I,I}(\pdico_I \transp \pdico_I)^{-1}.\label{c2_neumann}
\end{align}
The first two terms on the right hand side are each other's transpose, so substituting the above into \eqref{c2_EGYa} yields 
\begin{align}
     \|\E[\indi_{\mathcal{G}}(I)Y] \|  &\leq 2 \cdot \big\| \E_{\mathcal{G}} \big[\weights^{-1} R_I\transp \hol_{I,I}(\pdico_I \transp \pdico_I)^{-1}R_I \weights^{-1}\big] \big\| \cdot \|D_{ \alpha }^{-1}\| \notag \\
     &     + \|D_ {\alpha} ^{-1}\|\cdot \big\| \E_{\mathcal{G}} \big[\weights^{-1} R_I \transp (\pdico_I \transp \pdico_I)^{-1}\hol_{I,I}\transp \hol_{I,I}(\pdico_I \transp \pdico_I)^{-1}R_I \weights ^{-1}\big] \big\|\cdot \|D_{ \alpha } ^{-1}\| .\label{c2_EGYb}
\end{align}
To estimate the first term we repeat the steps in \eqref{bound_c1_2HG}, noting that for all $I$ we have $\|\weights^{-1} R_I\transp \hol_{I,I} R_I \weights^{-1} \| \leq \rho \alphamin^2/2$. Using \eqref{bound_c1a} as well as \eqref{bound_c1b} we get
\begin{align}
    \| \E_{\mathcal{G}} \weights^{-1} R_I\transp \hol_{I,I}(\pdico_I \transp \pdico_I)^{-1}R_I \weights^{-1}]\| &\leq 
    \P(\mathcal{G}^c) \cdot \rho \alphamin^2/2   + \| \E[\weights^{-1} R_I\transp \hol_{I,I} R_I \weights^{-1}  ] \| \notag \\
& \qquad \quad +\| \E_{\mathcal{G}}[\weights^{-1} R_I\transp \hol_{I,I}(\pdico_I \transp \pdico_I)^{-1}H_{I,I} R_I  \weights^{-1}  ] \|\notag\\
&\hspace{-1.3cm}\leq  \P(\mathcal{G}^c) \cdot \rho \alphamin^2/2 + 3 \cdot  ( \|\dico \weights\| \cdot \delta^2/2 + \| Z \weights\|) \cdot\| \pdico \weights\|\notag\\
& \qquad \quad+4\cdot \big(\|Z\weights\| + \delta \big)\cdot \|\pdico\weights\|^2. \label{bound_c2a}
\end{align}
Before we estimate the second term note that for $I\in \mathcal{G}$ we can use \eqref{hol_expressions} to bound $\|\hol_{I,I}(\pdico_I \transp \pdico_I)^{-1}\|$ in two different ways, either as 
\begin{align}
\|\hol_{I,I}(\pdico_I \transp \pdico_I)^{-1}\| &= \|(\dico - \pdico D_\alpha)_I\transp \pdico_I  (\pdico_I \transp \pdico_I)^{-1}\|\notag \\
&= \| \dico_I (\pdico_I^\dagger)\transp + (D_\alpha)_{I,I}\| \leq \sqrt{\tfrac{1+\vartheta}{1-\vartheta}} + 1 \leq \sqrt{5/3} +1\leq 7/3,
\end{align}
or recalling that on $\mathcal{G}$ we have $\|Z_I\|^2= \|Z_I\transp Z_I\|\leq 2\delta^2 \log(n K \rho /\deltamin) = 2\delta^2/\nu^2$ as
\begin{align}
\|\hol_{I,I}(\pdico_I \transp \pdico_I)^{-1}\| &= \|(\pdico \holdiag -Z)_I\transp  \pdico_I  (\pdico_I \transp \pdico_I)^{-1}\| \notag \\
&= \| \holdiag_{I,I} -  Z_I \transp (\pdico_I^\dagger)\transp \| \leq \eps^2/2 + \delta/\nu  \cdot \sqrt{2} \leq \sqrt{2} \delta/\nu \cdot(\nu/2 + 1) \leq 2\delta/\nu.
\end{align}
We next split the second term using \eqref{neumann_Hright}, apply Theorem~\ref{bmb} with the abbreviation $\Gamma =\min\{7/3, 2\delta/\nu\}$ and use Corollary~\ref{the_corollary}\eqref{E_RtNNtR_gen} or rather \eqref{E_RtNtNR} and \eqref{E_RtHHR} to get
\begin{align}
\big\| \E_{\mathcal{G}} \big[&\weights^{-1} R_I \transp (\pdico_I \transp \pdico_I)^{-1}\hol_{I,I}\transp \hol_{I,I}(\pdico_I \transp \pdico_I)^{-1}R_I \weights ^{-1}\big] \big\|\notag \\
&\leq \big\| \E \big[\weights^{-1} R_I \transp \hol_{I,I}\transp \hol_{I,I}R_I \weights ^{-1}\big] \big\| \notag \\
&\quad + 2  \big\| \E_{\mathcal{G}} \big[\weights^{-1} R_I \transp \hol_{I,I}\transp \cdot  \hol_{I,I}(\pdico_I \transp \pdico_I)^{-1} \cdot H_{I,I} R_I \weights ^{-1}\big] \big\| \notag\\
&\quad \quad + \big\| \E_{\mathcal{G}} \big[\weights^{-1} R_I \transp H_{I,I} \cdot (\pdico_I \transp \pdico_I)^{-1} \hol_{I,I}\transp \hol_{I,I}(\pdico_I \transp \pdico_I)^{-1} \cdot H_{I,I} R_I \weights ^{-1}\big] \big\| \notag\\
&\leq \big\| \E \big[\weights^{-1} R_I \transp \hol_{I,I}\transp \hol_{I,I}R_I \weights ^{-1}\big] \big\|\notag \\
& \quad + 2\big\| \E \big[\weights^{-1} R_I \transp \hol_{I,I}\transp \hol_{I,I}R_I \weights^{-1}\big]\big\|^{1/2} \cdot \Gamma \cdot \big\| \E \big[\weights^{-1} R_I \transp H_{I,I}H_{I,I} R_I \weights ^{-1}\big] \big\|^{1/2} \notag\\
&\quad \quad + \big\| \E \big[\weights^{-1} R_I \transp H_{I,I}H_{I,I} R_I \weights ^{-1}\big] \big\| \cdot \Gamma^2 \notag\\
&= \left(\big\| \E \big[\weights^{-1} R_I \transp \hol_{I,I}\transp \hol_{I,I}R_I \weights^{-1}\big]\big\|^{1/2} + \Gamma \,\big\| \E \big[\weights^{-1} R_I \transp H_{I,I}H_{I,I} R_I \weights ^{-1}\big] \big\|^{1/2}\right)^2 \notag\\
&\leq 2\left( \|\dico \weights\|\cdot \delta^2/2 + \|Z \weights\| + \min\{7/3, 2\delta/\nu\} \cdot  \|\pdico\weights\|  \right)^2.\label{bound_c2b}
\end{align}
Substituting \eqref{bound_c2a} and \eqref{bound_c2b} into \eqref{c2_EGYb} and this in turn into \eqref{c2_EY} yields
\begin{align}
 \| \E[ \hat{Y} ] \| 
   \leq  2\rho \cdot [ \P(\mathcal{H}^c) +  \P(\mathcal{G}^c)] +  6/\alphamin \cdot  \big( \|\dico \weights\| \cdot \delta^2/2 + \| Z \weights\|\big) &\cdot\| \pdico \weights\|\notag\\
 +8/\alphamin \cdot \big(\|Z\weights\| + \delta\big)&\cdot \|\pdico\weights\|^2 \notag \\
  +2/\alphamin^2 \cdot \big(\|\dico \weights\| \cdot \delta^2/2 + \| Z \weights\| +\min\{7/3, 2\delta/\nu\} &\cdot  \|\pdico\weights\|  \big)^2.
 \end{align}
We next show that the expression above is bounded by $\Delta/6$. In both regimes we have by Lemma~\eqref{prob_bounds} that $2\rho \cdot [ \P(\mathcal{H}^c) +  \P(\mathcal{G}^c)] \leq \deltamin/16$. In the first regime, if $\delta>\delta_\circ=\gamma \nu^2/C$ and therefore $\min\{\delta_\circ, \delta\} =\delta_\circ$, we use that $\delta\leq \sqrt{2}$, $\max\{\|\pdico \weights \|, \|\dico \weights \| \}\leq \alphamin\gamma \nu/(4C)$ and thus $\|Z \weights \|\leq \alphamin\gamma \nu/(2C)$ to get 
\begin{align*}
 \| \E[ \hat{Y} ] \| 
   \leq  \frac{\deltamin}{16} + \frac{\gamma \nu^2}{C} \cdot \frac{\gamma}{16C}  \big[ 18 \alphamin  + 12\alphamin + 2 \cdot (16/3)^2\big] \leq \frac{\deltamin}{16} + \delta_\circ \cdot \frac{87\gamma}{16C}\leq \frac{\Delta}{6}. 
 \end{align*}
Conversely in the second regime, when $\delta\leq \delta_\circ$ and therefore $\min\{\delta_\circ, \delta\} =\delta$, we use that $\|Z \weights \|\leq \delta \leq \delta_\circ =\gamma \nu^2/C$ and $\|\pdico \weights \|\leq \alphamin\gamma \nu/C$ as well as $\nu \leq 1/2$ and $\alphamin \geq 1-\delta_\circ^2/2 \geq 17/18$ to get
\begin{align*}
 \| \E[ \hat{Y} ] \| 
   \leq  \frac{\deltamin}{16} +  \delta \cdot \frac{\gamma}{2C}  \left[ 6 + \frac{6\alphamin\gamma^2}{64C^2}    
 +\frac{8\alphamin \gamma}{4C} 
  + \Big(\frac{\gamma^2 }{64C^2} + \frac{1}{\alphamin}+ \frac{ 2\gamma }{C}   \Big)^{\!2} \right]
  \leq  \frac{\deltamin}{16}  +  \delta \cdot \frac{4\gamma}{C} \leq \frac{\Delta}{6}.
 \end{align*}
 As before we arrive at the desired bound after applying the matrix Bernstein inequality~\ref{bernstein} for $t= \Delta /16$ with $r = \frac{3}{4}  \rho$ and $m =  \Delta /6$ and some simplifications.
\end{proof}
Now we turn to bounding individual columns of the random matrices treated in the last two lemmas. 
\begin{lemma}\label{lemma_claim3}
Under the conditions of Theorem~\ref{the_theorem} we have
\begin{align*}
  \P \Big( \| \dico \mata (D_{\pi \cdot \alpha \cdot \coefficient})^{-1}e_\ell - \atom_{\ell}  \| >   \Delta/8 \Big) \leq 28 \exp{\left(- \frac{N ( \Delta/16)^2 }{2 \rho^2+ \rho   \Delta/16 }\right)} .
\end{align*} 
\end{lemma}
\begin{proof}
As in the matrix case the idea is to write the vector whose norm we want to estimate as sum of independent random vectors based on the signals $y_n$ and use Bernstein's inequality. To this end we define for a fixed index $\ell$ the random vectors
\begin{align*}
\hat Y_n & :=  \big[y_n y_n\transp ( \pdico_{\hat{I}_n}^{\dagger})\transp R_{\hat{I}_n}  (D_{\pi \cdot \alpha \cdot \coefficient})^{-1} \cdot \indi_{\mathcal{B}(y_n)}(\hat I_n) - \dico \big]e_{\ell},\\
 Y_n & :=\big[y_n y_n \transp  (\pdico_{I_n}^{\dagger})\transp R_{I_n}  (D_{\pi  \cdot \alpha \cdot \coefficient})^{-1}\cdot \indi_{\mathcal{B}(y_n)}(I_n) -  \dico \diag(\one_{I_n})\weightss^{-1} \big] e_\ell.\nonumber
\end{align*}
Note that we can obtain $\hat Y_n, Y_n$ by multiplying the analogue matrices in the proof of Lemma~\ref{lemma_claim1} from the right by $\weights^{-1} e_\ell$. Following the proof strategy of Lemma~\ref{lemma_claim1} with the necessary changes, we first bound the $\ell_2$-norm of the random vectors $\hat Y_n, Y_n$ as
\begin{align*}
\max\{ \| \hat Y_n \|, \| Y_n \|\}  &\leq \kappa S^2 c_{\max}^2\|\|\weightss^{-1}\|  \|\coeff^{-1}\| \|\coeffm^{-1}\| + S \|\weightss^{-1}\| \leq 3\rho/4 =: r . 
\end{align*}
Repeating the procedures in \eqref{c1_HGsplit} we next get 
\begin{align}
    \| \E[\hat Y] \| &\leq  [2\P(\mathcal{H}^c) + \P(\mathcal{G}^c)]\cdot \rho +\|\E[\indi_{\mathcal{G}}(I)Y]\| \leq \deltamin/32 +\|\E[\indi_{\mathcal{G}}(I)Y]\|, \label{c3_EY}
\end{align} while repeating the procedures in \eqref{c1_EG_signs_coeff}/\eqref{c2_EY}, and using the expression from \eqref{dico_transp_pdico_dagger}  $\dico_I \transp (\pdico_I^{\dagger})\transp  = (\coeffm)_{I,I} +\hol_{I,I} (\pdico_I \transp \pdico_I)^{-1}=(\coeffm)_{I,I} + \hol_{I,I} + \hol_{I,I}(\pdico_I \transp \pdico_I)^{-1}H_{I,I}$ yields
\begin{align}
\|\E[\indi_{\mathcal{G}}(I)Y]\| &=  \big\|  \E_{\mathcal{G}} \big [\dico R_I \transp \big(\dico_I \transp \pdico_I^{\dagger\ast} -(\coeffm)_{I,I}  \big)R_I e_\ell\big] \big\| / (\alpha_\ell \pi_\ell) \notag \\
  &\leq \big\| \E_{\mathcal{G}}  \big [\dico R_I \transp \hol_{I,I} (\pdico_I \transp \pdico_I)^{-1}R_I e_\ell\big] \big\|/ (\alphamin  \pi_\ell) \nonumber \\
     &\leq  \left( \big\|  \E_{\mathcal{G}}  \big [\dico R_I \transp \hol_{I,I} R_I e_\ell\big] \big\| +  \big\| \E_{\mathcal{G}}  \big [\dico R_I \transp \hol_{I,I}(\pdico_I \transp \pdico_I)^{-1}H_{I,I} R_I e_\ell\big] \big\|\right)/(\alphamin  \pi_\ell).\label{c3_EGY}
\end{align}
We next use the decomposition $\mathbb{I}=\mathbb{I}_{\ell^c} + e_{\ell} e_{\ell}\transp$. Note that for any diagonal matrix $D$ we have $\mathbb{I}_{\ell^c}D e_\ell = 0$ and that for any matrix $V_{\ell\ell}= e_\ell V e_\ell$. Since $\hol_{I,I}$ has a zero diagonal and we have $\hol_{I,I} = R_I(\pdico \holdiag -Z)\transp \pdico R_I
\transp=R_I(\dico \holdiag -Z\coeffm)\transp \pdico R_I\transp$, Corollary~\ref{the_corollary}\eqref{E_WRtRVt} yields for the first term
\begin{align}
   \|\dico  \E_{\mathcal{G}}  \big [ R_I \transp \hol_{I,I} R_I e_\ell\big] \big\| 
   &= \| \E_{\mathcal{G}}  \big [\dico \mathbb{I}_{\ell^c} R_I \transp R_I(\dico \holdiag -Z\coeffm)\transp \, \indi_I(\ell) \big] \cdot \patom_\ell \big\|\notag\\
   &\leq  \pi_\ell \cdot  \| \dico \mathbb{I}_{\ell^c} \weights\| \cdot \|( \dico \holdiag -Z\coeffm)\weights\| \cdot \| \patom_\ell \big\|\notag \\
   &\leq \pi_\ell \cdot  \| \dico \weights\| \cdot  ( \| \dico \weights\| \cdot \eps^2/2 +\| Z\weights\|). \label{c3_b1}
\end{align}
Before we estimate the second term with the same split, note that applying Corollary~\ref{the_corollary}\eqref{E_RtNNtR_indl_gen} to $H,\hol,\hol \transp$ and using that $\max\{\|\pdico\weights\|,\|\dico \weights\| \}\leq \nu/C \leq 1/8$ yields the following three bounds, derived in detail in Appendix~\ref{sec:supp_proof_corollary},
\begin{align}
    \| \E [ \weights ^{-1} \mathbb{I}_{\ell^c} R_I \transp H_{I,I} H_{I,I}\transp R_I\mathbb{I}_{\ell^c} \weights ^{-1}  \cdot \indi_I(\ell)]\| 
     & \leq  9 \cdot \pi_\ell \cdot \max\{\mu(\pdico),\| \pdico \weights\| \}^2, \label{E_RtHHR_indl}\\
     \| \E [ \weights ^{-1} \mathbb{I}_{\ell^c} R_I \transp \hol_{I,I} \hol_{I,I}\transp R_I\mathbb{I}_{\ell^c} \weights ^{-1}  \cdot \indi_I(\ell)]\| 
     & \leq 9\cdot \pi_\ell \cdot \delta^2,\label{E_RtNNtR_indl}  \\
     \| \E [ \weights ^{-1} \mathbb{I}_{\ell^c} R_I \transp \hol_{I,I}\transp \hol_{I,I} R_I\mathbb{I}_{\ell^c} \weights ^{-1}  \cdot \indi_I(\ell)]\| 
     & \leq 9\cdot \pi_\ell \cdot \delta^2. \label{E_RtNtNR_indl}  
\end{align}
Further, recall that $H_{I,\ell} = R_I H e_\ell$ and $\hol_{\ell,I} = e_\ell\transp \hol R_I \transp$. So applying Theorem~\ref{bmb} and using \eqref{E_RtNNtR_indl} as well as twice Corollary~\ref{the_corollary}\eqref{E_WRtRVt}, we get 
\begin{align}
    \big\| \E_{\mathcal{G}} \big [\dico & R_I \transp \hol_{I,I}(\pdico_I \transp \pdico_I)^{-1}H_{I,I} R_I e_\ell\big] \big\| \notag \\
    &\leq \| \atom_\ell \| \cdot \big\| \E_{\mathcal{G}}\big [ \hol_{\ell,I}\cdot(\pdico_I \transp \pdico_I)^{-1}\cdot H_{I,\ell} \,\indi_I(\ell)\big] \big\|\notag \\
    & \hspace{2cm} + \| \dico \weights \| \cdot \big\|\E_{\mathcal{G}}\big [ \weights^{-1}\mathbb{I}_{\ell^c} \hol_{I,I}\cdot(\pdico_I \transp \pdico_I)^{-1}\cdot H_{I,\ell} \,\indi_I(\ell)\big] \big\|\notag \\
    & \leq \left(\| \E [ \hol_{\ell,I} \hol_{\ell,I} \transp \indi_I(\ell )]\|^{1/2} + \| \dico \weights \| \cdot\|\E [ \weights^{-1}\mathbb{I}_{\ell^c} \hol_{I,I}\hol_{I,I}\transp\mathbb{I}_{\ell^c} \weights^{-1} \indi_I(\ell )]\|^{1/2}\right) \notag \\
    & \hspace{7cm}\cdot 4/3 \cdot \| \E[ H_{I,\ell}\transp H_{I,\ell}  \indi_I(\ell )]\|^{1/2} \notag \\
    &\leq \left(\sqrt{\pi_\ell} \cdot \| e_\ell\transp \hol\weights \| + \| \dico \weights \| \cdot \sqrt{\pi_\ell}\cdot 3\delta \right) \cdot 4/3 \cdot  \sqrt{\pi_\ell} \cdot \|e_\ell\transp H\weights \|\notag \\
    &\leq  4/3 \cdot \pi_\ell  \cdot \delta \cdot (\| \pdico\weights \| + 3 \| \dico \weights \|)  \cdot \|\pdico \weights \| .\label{c3_b2}
\end{align}
Plugging the last two bounds back into~\eqref{c3_EGY} and~\eqref{c3_EY} yields
\begin{align} 
    \| \E[ \hat{Y}] \|  \leq \deltamin/32  &+  1/\alphamin\cdot  \| \dico \weights\| \cdot  ( \| \dico \weights\| \cdot \eps^2/2 +\| Z\weights\|) \notag\\
    &\qquad + 4/3 \cdot \delta/\alphamin \cdot \|\pdico \weights \| \cdot  (\| \pdico\weights \| + 3 \| \dico \weights \|)   \nonumber.
\end{align}
We next proceed as in the last proof. We always have $\| \dico \weights \| \leq \alphamin \gamma \nu/ (4C)$. If $\delta>\delta_\circ$, thus $\min\{\delta,\delta_\circ\} = \delta_\circ$, we use that additionally we have $\| \pdico \weights \| \leq \alphamin \gamma \nu/ (4C)$ and $\| Z \weights \| \leq \alphamin \gamma \nu/ (2C)$ to get
\begin{align} 
    \| \E[ \hat{Y}] \| & \leq \frac{\deltamin}{32}  + \frac{\gamma \nu^2}{C}\cdot \frac{\alphamin \gamma}{16C} \cdot ( 1+2  + 4/3 \cdot \sqrt{2} \cdot 4) \nonumber\leq \frac{\deltamin}{32}  + \delta_\circ \cdot \frac{11 \alphamin \gamma}{16C} \leq \frac{\Delta}{16}.
\end{align}
Conversely, if $\delta\leq\delta_\circ$, thus $\min\{\delta,\delta_\circ\} = \delta$, we use that $\| Z \weights \| \leq \delta $ together with $\| \pdico \weights \| \leq \alphamin \gamma \nu/ C$ and $\nu\leq 1/3$ to get
\begin{align} 
    \| \E[ \hat{Y}] \|  \leq \frac{\deltamin}{32}   &+  \delta \cdot\frac{ \gamma \nu}{4C} \cdot \left(   \frac{\alphamin \gamma \nu}{4C}  \cdot \delta/2 + 1  + 4/3  \cdot  \frac{7\alphamin \gamma \nu}{C}   \right)\leq \frac{\deltamin}{32}   + \delta \cdot \frac{ 2\gamma \nu}{4C} \nonumber \leq \frac{  \Delta}{16}. \nonumber
\end{align}
Finally an application of the vector Bernstein inequality~\ref{bernstein} for $t=m=  \Delta /16$ and $r = 3 \rho/4$ and some simplifications yield the desired bound.
\end{proof}
Now to the grand final, showing that also the fourth inequality used in the proof of the main theorem is satisfied with high probability. 
\begin{lemma}\label{lemma_claim4}
Under the conditions of Theorem~\ref{the_theorem} for $\Lambda = \max\left\{\|\pdico \weights \|, \frac{\alphamin \gamma\nu}{4C} \right\}$ 
\begin{align*}
    \P \left( \Lambda \cdot \| \mathbb{I}_{\ell^c}\,(D_{\sqrtpi  \cdot \alpha })^{-1} \matb (D_{\pi  \cdot \alpha \cdot \coefficient})^{-1} \, e_\ell \| > \frac{3  \Delta}{16} \right) 
    \leq 28 \exp{\left(- \frac{N ( \Delta/16)^2 }{2\rho^2+ \rho \Delta/16 }\right)}.
\end{align*}
\end{lemma}
\begin{proof}
As usual we rewrite the vector to bound as sum of random vectors based on the signals $y_n$ and use Bernstein's inequality. Thus we define 
\begin{align}
\hat{Y}_n :&= \Lambda \cdot \mathbb{I}_{\ell^c} \,(D_{\sqrtpi  \cdot \alpha }  )^{-1}R_{\hat{I}_n}\transp \pdico_{\hat{I}_n}^{\dagger} y_n y_n \transp \pdico_{\hat{I}_n}^{\dagger \ast} R_{\hat{I}_n} (D_{\pi  \cdot \alpha \cdot \coefficient} )^{-1} \indi_{\mathcal{B}(y_n)}(\hat{I}_n) \, e_\ell \notag
\end{align}
and its counterpart $Y_n$ by simply replacing in the above $\hat I_n$ by $I_n$. Since for any diagonal matrix $D$ we have $\mathbb{I}_{\ell^c} \,D \,e_\ell = 0$ we can again obtain $\hat Y_n, Y_n$ from the corresponding matrices in the proof of Lemma~\ref{lemma_claim2}, this time by multiplying from the left by $\mathbb{I}_{\ell^c}$ and from the right by $\weights^{-1} e_\ell$.
Following the usual proof strategy we first bound the $\ell_2$-norm of the random vectors $\hat Y_n, Y_n$ as 
\begin{align*}
\max\{ \| \hat Y_n \|, \| Y_n \|\}  &\leq \kappa^2\|y \|^2 \|\coeffm^{-2}\|\|\coeff^{-1}\| \| \weightss^{-3/2}\| \leq 3 \rho/4 =: r,
\end{align*}
while for the expectation we get similar to \eqref{c1_EG_signs_coeff} and \eqref{c2_EGYa}
\begin{align}
    \| \E[\hat Y] \| 
   &\leq \deltamin/32  + \Lambda \cdot \pi_{\ell}^{-1} \cdot \| D_{\alpha }^{-2}\| \cdot \big\| \E_{\mathcal{G}}\big [ \weights^{-1} \mathbb{I}_{\ell^c} R_I \transp \big(\pdico_I^{\dagger} \dico_I \dico_I \transp \pdico_I^{\dagger\ast}  \big)R_I  e_{\ell} \big] \big\|.
   \label{c4_EY}
\end{align}
Using a decomposition as in~\eqref{c2_neumann} we get
\begin{align}
     \mathbb{I}_{\ell^c} R_I \transp \big(\pdico_I^{\dagger} \dico_I \dico_I \transp \pdico_I^{\dagger\ast}  \big)R_I  e_{\ell}  = -\mathbb{I}_{\ell^c} R_I \transp &  \hol_{I,I}(\pdico_I \transp \pdico_I)^{-1} R_I e_{\ell} -  \mathbb{I}_{\ell^c} R_I \transp (\pdico_I \transp \pdico_I)^{-1}  \hol_{I,I} \transp \, R_I e_{\ell} \notag\\
    & + \mathbb{I}_{\ell^c} R_I \transp  (\pdico_I \transp \pdico_I)^{-1} \hol_{I,I}\transp\hol_{I,I}(\pdico_I \transp \pdico_I)^{-1}  R_I e_{\ell} .\label{c4_neumann}
\end{align}
The expectation corresponding to the first term can be obtained by replacing $\dico$ with $\weights^{-1}\mathbb{I}_{\ell^c}$ in \eqref{c3_EGY}. Going through \eqref{c3_b1} and \eqref{c3_b2} with the same change yields
\begin{align}
 \| \E_{\mathcal{G}}[\weights^{-1}\mathbb{I}_{\ell^c} R_I \transp   \hol_{I,I}&(\pdico_I \transp \pdico_I)^{-1} R_I e_{\ell}]\|\notag\\
 &\leq\pi_\ell \cdot ( \| \dico \weights\| \cdot \eps^2/2 +\| Z\weights\|+ 4  \delta   \cdot \|\pdico \weights \|). \label{c4_1a}
\end{align}
To estimate the second term note that $\hol_{I,I} \transp \, R_I e_{\ell} = R_I \mathbb{I}_{\ell^c} \hol \transp e_\ell \indi_I(\ell )$ since $\hol$ has a zero diagonal. So using the identity \eqref{neumann_Hleft}, Theorem~\ref{bmb}, Corollary~\ref{the_corollary}\eqref{E_WRtRVt} and \eqref{E_RtNNtR_indl_gen} or rather \eqref{E_RtHHR_indl} we get for the second term
\begin{align}
     \| \E_{\mathcal{G}}[\weights^{-1} &\mathbb{I}_{\ell^c} R_{I} \transp (\pdico_I \transp \pdico_I)^{-1}  \hol_{I,I} \transp R_I  \, e_{\ell} ] \|  \notag\\
     &\leq \| \E_{\mathcal{G}}[\weights^{-1} \, \mathbb{I}_{\ell^c}\, R_{I} \transp (\pdico_I \transp \pdico_I)^{-1} R_I \mathbb{I}_{\ell^c} \weights^{-1} \indi_I(\ell )]\| \cdot \|\weights \hol \transp e_\ell \| \nonumber \\
     &\leq \eps \cdot \| \pdico \weights\| \cdot \big(\| \E_{\mathcal{G}}[\weights^{-1} \mathbb{I}_{\ell^c}\, R_{I} \transp  R_I \mathbb{I}_{\ell^c}\weights^{-1}\,\indi_I(\ell ) ] \| \notag\\
     & \hspace{3cm} + \| \E_{\mathcal{G}}[\weights^{-1} \mathbb{I}_{\ell^c} R_{I} \transp H_{I,I} \cdot (\pdico_I \transp \pdico_I)^{-1} \cdot \mathbb{I}_{\ell^c} R_I \weights^{-1}\,\indi_I(\ell ) ] \| \big)  \nonumber \\
      &\leq \delta \cdot \|\pdico \weights\| \cdot \big(\pi_\ell + \| \E[\weights^{-1} \mathbb{I}_{\ell^c} R_{I} \transp H_{I,I} H_{I,I}\transp R_{I}\mathbb{I}_{\ell^c} \weights^{-1}\indi_I(\ell ) ]\|^{1/2} \cdot 4/3 \cdot \sqrt{\pi_\ell}\big) \nonumber\\
      &\leq \pi_\ell\cdot \delta \cdot \|\pdico \weights\| \cdot \big(1 + 4\cdot \max\{\mu(\pdico),\| \pdico \weights\| \} ).\label{c4_1b}
\end{align}
As probably feared the third term in \eqref{c4_neumann} requires further decomposition. Again we split the inverse  $(\pdico_I \transp \pdico_I)^{-1}$ into $(\pdico_I \transp \pdico_I)^{-1} = \mathbb{I} + (\pdico_I \transp \pdico_I)^{-1} H_{I,I} =  \mathbb{I} +  H_{I,I}(\pdico_I \transp \pdico_I)^{-1}$. Recalling that for $I \in \mathcal{G}$ we have $\|\hol_{I,I} (\pdico_I \transp \pdico_I)^{-1}  \|\leq \min \{7/3, 2\delta/\nu\}=\Gamma$ and applying Theorem~\ref{bmb} to all four resulting terms yields
\begin{align}
    \| \E_\mathcal{G} [ \weights^{-1} &\mathbb{I}_{\ell^c} R_I \transp  \hol_{I,I}\transp \cdot \hol_{I,I}  R_I \, e_{\ell} ] \| \notag \\
   &\leq \| \E [\weights^{-1}  \mathbb{I}_{\ell^c} R_I \transp \hol_{I,I}\transp \hol_{I,I} R_I   \mathbb{I}_{\ell^c} \weights^{-1} \indi_I(\ell ) ]\|^{1/2} \cdot \| \E [ \hol_{I,\ell}\transp  \hol_{I,\ell} \indi_I(\ell )]\|^{1/2} \nonumber \\
    \| \E_\mathcal{G} [ \weights^{-1} &\mathbb{I}_{\ell^c}  R_I \transp H_{I,I} \cdot (\pdico_I \transp \pdico_I)^{-1}  \hol_{I,I} \transp\cdot  \hol_{I,I} R_I \, e_{\ell} ] \| \notag\\
    & \leq 
    \| \E[ \weights^{-1}   \mathbb{I}_{\ell^c}  R_I \transp H_{I,I}H_{I,I}\transp R_I  \mathbb{I}_{\ell^c}  \weights ^{-1} \indi_I(\ell ) ]\|^{1/2}   
    \cdot \Gamma \cdot \|\E [ \hol_{I,\ell}\transp  \hol_{I,\ell} \indi_I(\ell )]\|^{1/2}  \nonumber \\
     \| \E_\mathcal{G}[ \weights^{-1}  &\mathbb{I}_{\ell^c}  R_I \transp \hol_{I,I} \transp \cdot \hol_{I,I} (\pdico_I \transp \pdico_I)^{-1} \cdot H_{I,I}  R_I  e_{\ell} ] \| \notag\\
    &\leq 
     \| \E[ \weights^{-1}   \mathbb{I}_{\ell^c}  R_I \transp \hol_{I,I} \transp \hol_{I,I} R_I  \mathbb{I}_{\ell^c}  \weights ^{-1} \indi_I(\ell ) ]\|^{1/2} 
    \cdot \Gamma \cdot  \| \E[ H_{I,\ell}\transp H_{I,\ell} \indi_I(\ell ) ]\|^{1/2} \nonumber \\
    \| \E_\mathcal{G} [\weights^{-1} &\mathbb{I}_{\ell^c}   R_I \transp   H_{I,I}\cdot ( \pdico_I \transp \pdico_I)^{-1}\hol_{I,I} \transp \hol_{I,I} ( \pdico_I \transp \pdico_I)^{-1}\cdot H_{I,I} R_I  e_{\ell} ] \| \notag\\
& \leq   \| \E[ \weights^{-1}   \mathbb{I}_{\ell^c}  R_I \transp H_{I,I}H_{I,I}\transp R_I  \mathbb{I}_{\ell^c}  \weights ^{-1} \indi_I(\ell )  ]\|^{1/2}   \cdot \Gamma^2
\cdot  \| \E[ H_{I,\ell}\transp H_{I,\ell} \indi_I(\ell ) ]\|^{1/2} \notag
\end{align}
Bounding the terms on the left hand side via Corollary~\ref{the_corollary}\eqref{E_RtNNtR_indl_gen} or rather 
\eqref{E_RtHHR_indl} and \eqref{E_RtNtNR_indl} and the terms on the right hand side via Corollary~\ref{the_corollary}\eqref{E_WRtRVt} we get 
\begin{align}
    \| \E_\mathcal{G} [ &\weights^{-1}\mathbb{I}_{\ell^c} R_I \transp  (\pdico_I \transp \pdico_I)^{-1} \hol_{I,I}\transp\hol_{I,I}(\pdico_I \transp \pdico_I)^{-1}  R_I e_{\ell}]\| \notag \\
    &\leq \pi_ \ell \cdot \big(3\delta + 3\Gamma \max\{\mu(\pdico),\|\pdico \weights \|\} \big)  \big(\|e_\ell\transp \hol \weights \| + \Gamma \|e_\ell\transp H\weights \|\big)\notag \\
    &\leq \pi_ \ell \cdot \big(3\delta + 3\Gamma \max\{\mu(\pdico),\|\pdico \weights \|\} \big) \big(\|\dico \weights\| \, \eps^2/2 + \|Z\weights\| + \Gamma \|\pdico \weights\|\big). \label{c4_1c}
\end{align}
Finally substituting \eqref{c4_1a}-\eqref{c4_1c} into \eqref{c4_EY} yields
\begin{align*}
   \| \E [\hat{Y}] \| \leq \frac{\deltamin}{32}  & +  \frac{\Lambda}{\alphamin^2} \Big( \| \dico \weights\| \, \frac{\eps^2}{2} +\| Z\weights\| +  \delta \, \|\pdico \weights\| \, \big(5+ 4 \max\{\mu(\pdico),\| \pdico \weights\| \} \big)\Big)\notag \\
  &  + \frac{3\Lambda}{\alphamin^2}  \Big(\delta + \Gamma \max\{\mu(\pdico),\|\pdico \weights \|\} \Big)  \Big(\|\dico \weights\| \, \frac{\eps^2}{2} + \|Z\weights\| + \Gamma \|\pdico \weights\|\Big).
\end{align*}
As before we distinguish between $\delta> \delta_\circ$, where $\min\{\delta, \delta_\circ\} = \delta_\circ = \gamma \nu^2/C$ and we can use 
$\|\pdico \weights\|\leq \alphamin \gamma \nu/(4C)=\Lambda$, $\mu(\pdico)\leq \alphamin \gamma \nu^2/(4C)$ as well as $\|Z \weights \| \leq \alphamin \gamma \nu/(2C)$ and $\Gamma \leq 7/3$, to get
\begin{align*}
   \| \E [\hat{Y}] \| \leq \frac{\deltamin}{32}  & +  \frac{ \gamma^2 \nu^2}{16C^2}  \Big[ 3 + \sqrt{2} \Big(5+ \frac{ \alphamin \gamma \nu}{C} \Big) + 16 \Big(\sqrt{2} +  \frac{ 7\alphamin \gamma \nu}{12C} \Big)\Big]\notag \leq \frac{\deltamin}{32}   +  \delta_\circ \frac{36 \gamma}{16C} \leq \frac{\Delta}{16},
\end{align*}
and $\delta\leq  \delta_\circ$, where $\min\{\delta, \delta_\circ\} = \delta$ and we can use
$\|\pdico \weights\|\leq \alphamin \gamma \nu/C = \Lambda$ and $\mu(\pdico)\leq 9 \nu^2/ (4C)$, meaning $\max\{\mu(\pdico),\| \pdico \weights\| \}\leq \nu\leq 1$ as well as $\|Z \weights \| \leq \delta $ and $\Gamma \leq 2\delta/\nu$, to again get
\begin{align*}
   \| \E [\hat{Y}] \| \leq \frac{\deltamin}{32}  & + \delta  \frac{\gamma \nu}{\alphamin C} \Big[1 + 9\delta + \frac{ 9\alphamin \gamma }{8C} \big(8\nu + \delta \nu + \delta^2\nu  + 16\delta \big)\Big]\leq \frac{\deltamin}{32} + \delta  \frac{2\gamma \nu}{\alphamin C}\leq \frac{\Delta}{16}.
\end{align*}
As before an application of the vector Bernstein inequality~\ref{bernstein} for $t= m = \Delta /16$ and $r = 3 \rho/4$ and some simplifications yield the desired bound.
\end{proof}
Note that for the probability bounds in the last four lemmas, we have used the simplified version of the vector/matrix Bernstein inequality from \cite{tr12}, stated in ~\ref{bernstein}, which bounds the variance appearing in the denominator via the quantity $r$. Using the same approach as above one could estimate terms of the form $\|\tilde Y_n \tilde Y_n\transp \|$ and $\|\tilde Y_n\transp \tilde Y_n \|$, where $\tilde Y_n$ is the centered version of $Y_n$, to improve the bounds on the variance and thus lower the number of training signals required in the main theorem, \cite{MINSKER2017111,tr12}.




\section{Technical results}\label{app:technical}
For convenience we first restate several bounds for random matrices used in the proofs of Lemma~\ref{prob_bounds} as well as Lemma~\ref{lemma_claim1}-\ref{lemma_claim4}. We then provide the proofs of Lemma~\ref{bmb} and Corollary~\ref{the_corollary} together with the detailed derivation of (\ref{E_RtHHR}-\ref{E_RtNtNR}) and (\ref{E_RtHHR_indl}-\ref{E_RtNtNR_indl}).

\subsection{Random matrix bounds} \label{sec:supp_rand_mat}
We first recall two bounds for sums of random matrices and vectors.
\begin{theorem}[Matrix Chernoff inequality~\cite{tr12}]\label{chern}
Let $X_1,...,X_N$ be independent random positive semi-definite matrices taking values in $\R^{d \times d}$. Assume that for all $n \in \{1,..., N\}$, $\| X_n \| \leq \eta$ a.s. and $\|\sum_{n = 1}^N \E[X_n] \| \leq \mu_{max}$. Then, for all $t \geq e \mu_{max}$, 
\[
\P\left( \| \sum_{n=1}^N X_n \| \geq t \right) \leq K \left( \frac{e \mu_{max}}{t} \right)^{\frac{t}{\eta}}.
\]
\end{theorem}
\begin{theorem}[Matrix resp. vector Bernstein inequality~\cite{tr12, MINSKER2017111}]\label{bernstein}
Consider a sequence $Y_1,...,Y_N$ of independent, random matrices (resp. vectors) with dimension $d \times K$ (resp. $d$). Assume that each random matrix (resp. vector) satisfies
\[
\|Y_n \| \leq r \quad \text{a.s.} \quad \text{and} \quad \|\E[Y_n] \| \leq m.
\]
Then, for all $t>0$,
\begin{align}
    \P \left(\| \frac{1}{N}\sum_{n=1}^N Y_n \| \geq m + t   \right) \leq \kappa  \exp \left( \frac{-N t^2}{2 r^2 + (r+m)t}\right),
\end{align}
where $\kappa = d + K$ for the matrix Bernstein inequality and $\kappa = 28$ for the vector Bernstein inequality.
\end{theorem}
Next we recall Theorem~3.1 from \cite{rusc21} which allows us to control the operator norm of a submatrix with high probability.
\begin{theorem}[Operator norm of a random submatrix~\cite{rusc21}] \label{them:opnorm}
Let $ \pdico$ be a dictionary and assume $I \subseteq\mathbb{K}$ is chosen according to the rejective sampling model with probabilities $p_1, \dots , p_K$ such that $\sum_{i = 1}^K p_i = S$. Further let $D_{p}$ denote the diagonal matrix with the vector $p$ on its diagonal. Then  
\[
\P \left(  \| \pdico_I\transp \pdico_I - \mathbbm{I} \| > \vartheta \right) \leq 216 K \exp{\left(- \min\left\{ \frac{\vartheta^2}{4e^2 \|  \pdico  D_{p} \pdico \transp\| } , \frac{\vartheta}{ 2 \mu(\pdico) } \right\}\right)}.
\]
\end{theorem}

To prove Theorem~\ref{bmb} we first state and prove the following lemma to bound sums of products of matrices.
\begin{lemma}[Sums of products of matrices ~\cite{sing_values_1},~\cite{sing_values_2}]
Let $A_n \in \R^{d_1 \times d_2}$, $B_n \in \R^{d_2 \times d_3}$, $C_n \in \R^{d_3 \times d_4}$. Then
\[
\left\| \sum_{n = 1 }^N A_n B_n C_n \right\| \leq \left\| \sum_{n = 1}^N A_n A_n\transp \right\|^{1/2} \max_{n}\| B_n \| \left\| \sum_{n = 1}^N C_n \transp C_n\right\|^{1/2}.
\]
\end{lemma}
\begin{proof}
Write
\begin{equation*}
\sum_{n = 1 }^N A_n B_n C_n =  
\begin{pmatrix}
 A_1 & A_2 & A_3 & . \\ . & . & & \\ .&&.& \\ .&&& . 
\end{pmatrix} 
\begin{pmatrix}
 B_1 & . & . & . \\ . & B_2 & & \\ .&&B_3& \\ .&&& .  
\end{pmatrix} 
\begin{pmatrix}
 C_1 & . &  .&  .\\  C_2 &. & &\\ C_3 & &. & \\ . & & & .
\end{pmatrix} .
\end{equation*}
Now the result immediately follows by applying the following properties of the operator norm $\| A B C\| \leq \|A \|\|B \|\| C\|  $, $\|A \| =  \| A A\transp \|^{1/2} $ and $\| C \| = \| C \transp C \|^{1/2}$.
\end{proof}
Lemma~\ref{bmb} is a straightforward consequence of the result above.
\begin{proof}[of Lemma~\ref{bmb}]
We want to show that for random matrices $A(I) \in \R^{d_1 \times d_2}$, $B(I) \in \R^{d_2 \times d_3}$, $C(I)\in \R^{d_3 \times d_4}$, where $I$ is a discrete random variable taking values in $\mathcal{I}$ and $\mathcal{G}\subseteq \mathcal{I}$ with $\max_{I\in\mathcal{G}} \|B(I)\| \leq \Gamma$ we have
 \begin{align*}
\|\E \left[ A(I) \cdot B(I) \cdot C(I) \cdot \indi_{\mathcal G}(I) \right] \| \leq \| \E \left[ A(I) A(I)\transp \right] \|^{1/2} \cdot \Gamma \cdot \| \E \left[ C(I) \transp C(I) \right] \|^{1/2}.
\end{align*}
Rewriting the expectation as a sum and applying the lemma above yields
\begin{align*}
\|\E \left[ A(I) B(I) C(I) \indi_{\mathcal G}(I) \right] \| &= \| \sum_{I \in \mathcal{G}}\P[I]^{1/2} A(I) B(I) C(I) \P[I]^{1/2}\|\\
& \leq \| \sum_{I \in \mathcal{G}}\P[I] A(I) A(I)\transp\|^{1/2} \cdot \Gamma \cdot \| \sum_{I \in \mathcal{G}}\P[I] C(I)\transp C(I)\|^{1/2}\\
&\leq \| \E \left[ A(I) A(I)\transp \right] \|^{1/2} \cdot\Gamma\cdot \| \E \left[ C(I) \transp C(I) \right] \|^{1/2},
\end{align*}
where in the last inequality we have used that the matrices $A(I) A(I)\transp$ and $C(I)\transp C(I)$ are positive semidefinite and that $\P[I]\geq 0$.
\end{proof}
\subsection{Proof of Corollary~\ref{the_corollary}}\label{sec:supp_proof_corollary}
Finally, we will turn to the proof of Corollary~\ref{the_corollary}. Note that in the uniformly distributed support model, which is equivalent to rejective sampling with uniform weights $p_i = S/K$, most estimates become trivial, since we have $\pi_i = \P(i \in I) = S/K =p_i$ and $\P(\{i,j\} \subseteq I) = \frac{S(S-1)}{K(K-1)}$.
This means that for a zero diagonal matrix $\hol$ we have
\begin{align}
    \E[ R_I \transp \hol_{I,I} R_I] = \E[ \hol \odot (\one_I \one_I \transp)]=  \hol \odot\E[\one_I \one_I \transp]=\hol \cdot \frac{S(S-1)}{K(K-1)}, \notag
\end{align}
so we simply get 
\begin{align}
  \| \E[ \weights^{-1}  R_I\transp \hol_{I,I} R_I  \weights ^{-1}]\| =  \| \hol \|  \cdot  \frac{S-1}{K-1} \leq  \| \hol \| \cdot  \frac{S}{K}  = \| \weights \hol \weights\|.\notag
\end{align}
Unfortunately, in the rejective sampling model with non-uniform weights $p_i$, these estimates become much more involved. In particular, we will heavily rely on the following theorem, collecting results from \cite{rusc22}.

\begin{theorem}\label{th_poisson}
Let $\P_B$ be the probability measure corresponding to the Poisson sampling model with weights $p_i < 1$ and $\P_S$ be the probability measure corresponding to the associated rejective sampling model with parameter $S$, $\P_S(I) = \P_B(I \mid |I|=S)$,  as in Definition~\ref{def_prob_models}. Further denote by $\E_S$ the expectation with respect to $\P_S$ and by $\pi_S$ the vector of first order inclusion probabilities of level $S$, meaning $\pi_S(i) = \P_S( i \in I)$ or equivalently $\pi_S = \E_S ( \one_I)$. 
We have
\begin{align*}
(1-\|p\|_\infty) \cdot p_i &\leq \pi_S(i) \leq 2 \cdot p_i,  \quad \text{if} \quad \textstyle{\sum_{k}} \,p_k = S, \tag{a} \label{p_leq_pi} \\ 
\pi_{S-1}(i) &\leq  \pi_S(i), \tag{b} \label{pi_Sm1_leq}\\
  \P_S( \{i,j\}\subseteq I) &\leq \pi_S(i) \cdot \pi_S(j), \quad \text{if} i \neq j .\tag{c}
\label{pi_ij_leq}
 \end{align*}
Further, defining for $L\subseteq [K]$ with $|L|<S$ the set $\mathcal L = \{I\subseteq [K] : L\subseteq I\}$, we have
 \begin{align*}
\E_{S} \big[ \one_{I\setminus L} \one_{I\setminus L } \transp \cdot \indi_\mathcal{L}(I) \big] \cdot \prod_{\ell \in L} [1-\pi_S(\ell)] & \preceq  \E_{S-|L|} [ \one_I \one_I\transp] \cdot \prod_{\ell \in L} \pi_S(\ell). \tag{d} \label{E_pi_matrix}
\end{align*}
Finally, if $\pi:=\pi_S$ satisfies $\|\pi\|_\infty < 1 $, then for any $K\times K$ matrix $A$ we have
\begin{align*}
    \| A \odot \E[\one_I \one_I \transp ] \| \leq \frac{1 + \|\pi\|_\infty}{(1 - \|\pi \|_\infty)^2} \cdot \|D_{\pi } [A -\diag(A)] D_{\pi}\| + \| \diag(A)D_{\pi}\|.  \tag{e} \label{th_poisson_odot}
\end{align*}
\end{theorem}
With these results in hand we can finally prove Corollary~\ref{the_corollary}.

\begin{proof}[of Corollary~\ref{the_corollary} including (\ref{E_RtHHR}-\ref{E_RtNtNR}) and (\ref{E_RtHHR_indl}-\ref{E_RtNtNR_indl})] \label{proof_the_corollary}
\phantom{bla}
\newline
\textbf{(a)} We want to show that for a matrix $\hol$ with zero diagonal we have
\begin{align*}
    \| \E[ \weights^{-1}  R_I\transp \hol_{I,I} R_I  \weights ^{-1}]\| & \leq 3\cdot \| \weights \hol \weights \|.
\end{align*}
Using the identities $A_{I,I} = R_I A R_I\transp$ and $R_I\transp R_I=\diag(\one_I)$, we can rewrite for a general matrix $A$ and a diagonal matrix $D$ \begin{align*}
D  R_I \transp A_{I,I} R_I  D 
= D R_I\transp R_I A  R_I \transp R_I D
&=  D \diag(\one_I) A  \diag(\one_I) D \\
&=  \diag(\one_I) D  A D  \diag(\one_I) = (D  A D) \odot (\one_I  \one_I\transp ).
\end{align*}
Using Theorem~\ref{th_poisson}\eqref{th_poisson_odot} and $\|\pi\|_\infty \leq 1/3$ we therefore get for $\hol$ with zero-diagonal
\begin{align*}
    \| \E[ \weights^{-1}  R_I \transp \hol_{I,I}  R_I  \weights ^{-1}]\|  
    & =\| \E[ (\weights^{-1} \hol \weights ^{-1})\odot (\one_I  \one_I\transp ) ] \|\\
    & =\|  (\weights^{-1} \hol \weights^{-1}) \odot\E[ \one_I  \one_I\transp  ] \| \\
    &\leq 3 \|D_{\pi} \weights^{-1} \hol \weights^{-1}  D_{\pi }\|= 3 \| \weights \hol \weights \|,
\end{align*}    
which proves (a). \hfill \textbf{(a)\checkmark} 

\noindent \textbf{(b)} We want to show that for a matrix $\hol$ with zero diagonal we have
\begin{align*}
\| \E[\weights^{-1}  R_I\transp \hol_{I,I} \hol_{I,I} \transp R_I  \weights^{-1}]\|  & \leq \tfrac{9}{2}  \cdot\| \weights \hol \weights\|^2 + \tfrac{3}{2} \cdot \max_k \|e_k\transp \hol \weights\|^2.
\end{align*}
We again rewrite the expression, whose expectation we need to estimate, as 
\begin{align*}
 R_I \transp \hol_{I,I} \hol_{I,I}\transp R_I&=   R_I \transp  R_I \cdot \hol \cdot R_I\transp  R_I \cdot \hol \transp \cdot R_I\transp R_I \\
 &=    \diag(\one_I) \cdot \hol  \cdot \diag(\one_I) \cdot \hol \transp \cdot  \diag(\one_I) \\
  & = [\hol \cdot \diag(\one_I) \cdot \hol \transp] \odot (\one_I\one_I\transp) \\
  &= \big( \sum_{k\in I} \hol_k \hol_k\transp \big) \odot (\one_I\one_I\transp)
  = \sum_{k \in I} (\hol_k \hol_k\transp) \odot (\one_I\one_I\transp).
\end{align*}
Since the $k$-th entry of $\hol_k$ and therefore both the $k$-th row and $k$-th column of $\hol_k \hol_k\transp$ are zero, we have $(\hol_k \hol_k\transp) \odot (\one_I\one_I\transp) = (\hol_k \hol_k\transp) \odot (\one_{I\setminus\{k\}} \one_{I\setminus\{k\}}\transp)$, yielding
\begin{align}
 R_I \transp \hol_{I,I} \hol_{I,I}\transp R_I& = \sum_{k \in I} (\hol_k \hol_k\transp) \odot (\one_{I\setminus\{k\}} \one_{I\setminus\{k\}}\transp) \notag \\
 &= \sum_{k} (\indi_I(k) \cdot \hol_k \hol_k\transp) \odot (\one_{I\setminus\{k\}} \one_{I\setminus\{k\}}\transp)\notag \\
 &=\sum_{k}  (\hol_k \hol_k\transp) \odot (\one_{I\setminus\{k\}} \one_{I\setminus\{k\}}\transp \cdot \indi_I(k) ). \label{HHstar_expansion}
\end{align}
Using the Schur Product Theorem, which says that for p.s.d matrices $A, P, \bar P$, with $P_{ij} , \bar P_{ij} \geq 0$ and $ P \preceq \bar P$ we have $ A \odot P \preceq A \odot \bar P$, together with Theorem~\ref{th_poisson}\eqref{E_pi_matrix} further leads to
\begin{align}
\E_S \left[ R_I \transp \hol_{I,I} \hol_{I,I}\transp R_I \right]&=
    \sum_{k} (\hol_k \hol_k\transp) \odot \E_{S} \big[ \one_{I\setminus\{ k\}} \one_{I\setminus \{k\}} \transp \cdot \indi_I(k) \big] \notag\\
    &\preceq \sum_{k}  ( \hol_k \tfrac{\pi_S(k)}{1-\pi_S(k)}  \hol_k\transp ) \odot \E_{S-1}[ \one_{I} \one_{I}\transp]  \nonumber \\
    &= \big( \sum_{k} \hol_k \tfrac{\pi_S(k)}{1-\pi_S(k)}  \hol_k\transp \big)  \odot \E_{S-1}[ \one_{I} \one_{I}\transp] \nonumber \\
    &=(\hol \diag(\tfrac{\pi_S}{1-\pi_S}) \hol\transp)\odot \E_{S-1}[ \one_{I} \one_{I}\transp] .\notag 
\end{align}
Abbreviating $M: = \hol \diag(\tfrac{\pi_S}{1-\pi_S}) \hol\transp$, and using Theorem~\ref{th_poisson} \eqref{th_poisson_odot} and~\eqref{pi_Sm1_leq} we get
\begin{align*}
\| \E_S \left[ \weights^{-1} R_I \transp \hol_{I,I} \hol_{I,I}\transp R_I \weights^{-1} \right] \| 
&= \| \weights^{-1}\, \E_S \left[  R_I \transp \hol_{I,I} \hol_{I,I}\transp R_I  \right] \weights^{-1} \| \\
&\leq \| (\weights^{-1} M \weights^{-1} )\odot \E_{S-1}[ \one_{I} \one_{I}\transp] \| \\
& \leq 3\|D_{\pi_{S-1}} [\weights^{-1} M \weights^{-1}  - \diag(\weights^{-1} M \weights^{-1} )] D_{\pi_{S-1}}\| \\
& \hspace{3.5cm}+ \| \diag(\weights^{-1} M \weights^{-1} )D_{\pi_{S-1}}\| \\
& \leq 3\| \weights M \weights  - \diag(\weights M \weights)\| + \| \diag(M) \|\\
& \leq 3\| \weights M \weights\| + \| \diag(M) \|,
\end{align*}
where in last inequality we have used that $\weights M \weights$ is positive semidefinite. Combining the inequality above with the bounds
\begin{align*}
\| \weights M \weights\| &= \| \weights \hol \weights \diag(\tfrac{1}{1-\pi})  \weights \hol\transp \weights\| \leq (1-\|\pi\|_\infty)^{-1} \| \weights \hol \weights\|^2,\\
\| \diag(M) \| & = \max_k e_k \transp \hol \weights \diag(\tfrac{1}{1-\pi})  \weights \hol\transp e_k \leq  (1-\|\pi\|_\infty)^{-1} \max_k \|e_k\transp \hol \weights\|^2,
\end{align*}
and our assumption that $\|\pi\|_\infty \leq 1/3$ leads to (b).\hfill \textbf{(b)\checkmark} \\

\noindent \textbf{(\ref{E_RtHHR}-\ref{E_RtNtNR})} We next specialise the general inequality from (b) to the three concrete matrices needed in the proofs of Lemma~\ref{lemma_claim1}-\ref{lemma_claim4}.
For the case $H = \mathbb{I}-\pdico\transp\pdico $, note that since $\weights \pdico\transp\pdico\weights $ is a positive semidefinite matrix we have 
\begin{align}\|\weights H \weights \| = \|\weights \pdico\transp\pdico\weights - \diag (\weights \pdico\transp\pdico\weights) \| \leq \|\weights \pdico\transp\pdico\weights\| = \|\pdico\weights\|^2. \notag
\end{align}
Further note that the k-th entry of $e_k \transp H$ is zero, meaning
\begin{align}
  \|e_k\transp H\weights\| =\|(e_k\transp - \patom_k \transp \pdico) \weights\| \leq\| \patom_k \transp \pdico \weights\|   \leq \|\pdico \weights\|.\notag
\end{align}
So, as long as $\|\pdico \weights\| \leq 1/3$, which holds in both regimes, we get 
\begin{align*}
\| \E[\weights^{-1}  R_I\transp H_{I,I} H_{I,I} \transp R_I  \weights^{-1}]\|  & \leq  9/2  \cdot\|\pdico\weights\|^4 + 3/2 \cdot \|\pdico \weights\|^2  \\
&\leq  (9/2 \cdot 1/3^2 + 3/2 ) \cdot \|\pdico \weights\|^2 = 2\|\pdico \weights\|^2.
\end{align*}
In case $\hol = (\pdico \holdiag -Z)\transp \pdico = (\dico \holdiag - Z D_\alpha )\transp \pdico$, we use 
\begin{align} 
\|\weights \hol \weights \| \leq \| (\dico \holdiag - Z D_\alpha ) \weights\|\cdot \| \pdico \weights \| &= (\|\dico \weights \holdiag + Z\weights D_\alpha \|) \cdot \| \pdico \weights \| \notag\\
&\leq   (\|\dico \weights\| \cdot \eps^2/2 + \| Z\weights\|) \cdot \| \pdico \weights \|. \notag
\end{align}
Further since $\holdiag_{kk} = \ip{\patom_k}{z_k}$ and $\mathbb{I} - \patom_k \patom_k\transp$ is an orthogonal projection we have
\begin{align} 
\|e_k \transp \hol \weights \| &\leq \| \patom_k \ip{\patom_k}{z_k} - z_k\|\cdot \| \pdico \weights \| \notag\\
&= \| (\patom_k \patom_k\transp -  \mathbb{I}) z_k\|\cdot \| \pdico \weights \| \leq  \|z_k\| \cdot \| \pdico \weights \| \leq \eps \cdot \| \pdico \weights\|. \notag
\end{align}
So, using that $\eps \leq \sqrt{2}$, as long as $\|\dico \weights\| \leq 1$, we get
\begin{align*}
2 \cdot \| \E[\weights^{-1}  R_I\transp &\hol_{I,I} \hol_{I,I} \transp R_I  \weights^{-1}]\|  \\
&\leq  (9\|\dico \weights\|^2 \cdot \eps^4/4  + 9 \eps^2 \| Z\weights\| + 9\| Z\weights\|^2  + 3 \eps^2 )\cdot\| \pdico \weights\|^2.\notag\\
& \leq  (9\cdot \eps^2/2  + 18 \eps \| Z\weights\| + 9\| Z\weights\|^2  + 3 \eps^2 )\cdot\| \pdico \weights\|^2 \\
&\leq (3\| Z\weights\|  + 3 \eps)^2 \cdot\| \pdico \weights\|^2.
\end{align*}
Finally note that $\|\weights \hol \weights \|  = \|\weights \hol \transp \weights \| $. Combining this with the bound 
\begin{align} 
\|e_k \transp \hol\transp  \weights \| &\leq \| \patom_k \| \cdot \| (\dico \holdiag - Z D_\alpha ) \weights\|  \leq \|\dico \weights\| \cdot \eps^2/2 + \| Z\weights\|.
\end{align}
we get that again as long as $\|\pdico \weights\| \leq 1/3$
\begin{align*}
\| \E[\weights^{-1}  R_I\transp \hol_{I,I}\transp \hol_{I,I}  R_I  \weights^{-1}]\| 
& \leq  (\|\dico \weights\| \cdot \eps^2/2 + \| Z\weights\|)^2 \cdot (9/2 \cdot \| \pdico \weights \|^2 + 3/2) \\
&\leq  (\|\dico \weights\| \cdot \eps^2/2 + \| Z\weights\|)^2 \cdot 2,
\end{align*}
which proves the third inequality.
\hfill \textbf{(\ref{E_RtHHR}-\ref{E_RtNtNR})\checkmark} \\

\noindent \textbf{(c)} We want to show that for a subset $\mathcal{G}$ of all supports of size $S$ and a pair of ${d\times K}$ matrices $W=(w_1 \ldots ,w_K)$ and $V=(v_1, \ldots ,v_K)$ we have
\begin{align*}
    \| \E[  W R_I\transp R_I V\transp \cdot \indi_I(\ell) \indi_{\mathcal{G}}(I)]\|&\leq \pi_\ell \cdot \left(\|W\weights\|\cdot \| V \weights \| +\|w_\ell\|\cdot \|v_\ell\|\right).
    \end{align*}For $\ell \in I$ we can rewrite $W R_I \transp R_I V\transp = W \diag(\one_{I})V\transp = W \diag(\one_{I\setminus\{\ell\}})V\transp + w_\ell v_\ell\transp$.
Using this split and that by Theorem~\ref{th_poisson}\eqref{pi_ij_leq} we have $\P(\{\ell,k\}\subseteq I) \leq \pi_\ell \pi_k$ yields
 \begin{align*}
\| \E [&W R_I\transp R_I V \transp \cdot \indi_I(\ell) \indi_{\mathcal{G}}(I) ]\| \\
&\leq  \| W \, \E \left[\diag(\one_{I\setminus\{\ell\}})\cdot \indi_I(\ell) \indi_{\mathcal{G}}(I)\right] V \transp  \| + \|w_\ell v_\ell\transp\| \cdot \E \left[\indi_I(\ell)\indi_{\mathcal{G}}(I) \right]\\
&\leq \| W \weights \| \cdot \| \weights^{-1} \E \left[\diag(\one_{I\setminus\{\ell\}})\cdot \indi_I(\ell) \indi_{\mathcal{G}}(I)\right] \weights^{-1}\|  \cdot \| \weights V \transp \|  + \|w_\ell v_\ell\transp\| \cdot \pi_\ell \\
&=\| W \weights \| \cdot \| V \weights \| \cdot \max_{k\neq \ell} \,\pi_k^{-1} \cdot \P(\{\ell,k\}\subseteq I \cap I\in \mathcal{G} ) +\|w_\ell\|\cdot  \| v_\ell\| \cdot \pi_\ell\\
&\leq \pi_\ell \cdot \left( \| W \weights \| \cdot \| V \weights \|  + \|w_\ell\|\cdot  \| v_\ell\|\right)
\end{align*}
 which completes the proof of (c).
\hfill \textbf{(c)\checkmark} \\

\noindent \textbf{(d)}
We prove that for a general matrix $\hol$ with zero diagonal we have
\begin{align}
\| \E [ \weights ^{-1} &\mathbb{I}_{\ell^c} R_I \transp \hol_{I,I} \hol_{I,I}\transp R_I\mathbb{I}_{\ell^c} \weights ^{-1}  \cdot \indi_I(\ell)]\| \notag \\
     & \leq \frac{\pi_\ell}{1-\pi_\ell} \left(3\| \weights \hol e_\ell \|^2 + \max_{k} \hol_{k\ell}^2 + \tfrac{9}{2} \| \weights \hol \weights\|^2 + \tfrac{3}{2} \max_k \|e_k\transp \hol\weights\|^2 \right). \label{cor_h_nonsym}
\end{align}
Using \eqref{HHstar_expansion} and the identity $\mathbb{I}_{\ell^c} = \diag(\one_{[K]\setminus\{\ell\}})$, we first rewrite
 \begin{align*}
\mathbb{I}_{\ell^c} R_I \transp \hol_{I,I} \hol_{I,I}\transp R_I \mathbb{I}_{\ell^c} \cdot  \indi_I(\ell)& = (R_I \transp \hol_{I,I} \hol_{I,I}\transp R_I) \odot (\one_{[K]\setminus\{\ell\}} \one_{[K]\setminus\{\ell\}} \transp) \cdot  \indi_I(\ell)\\
&= {\sum}_{k}  (\hol_k \hol_k\transp) \odot \big(\one_{I\setminus\{k,\ell\}} \one_{I\setminus\{k,\ell\}}\transp \cdot \indi_I(\ell)\indi_I(k)\big).
 \end{align*}
As before an application of the Schur Product Theorem and Theorem~\ref{th_poisson}\eqref{E_pi_matrix} leads to
  \begin{align*}
 \E_S[ \mathbb{I}_{\ell^c} &R_I \transp \hol_{I,I} \hol_{I,I}\transp R_I \mathbb{I}_{\ell^c} \cdot  \indi_I(\ell)] 
 = \sum_{k}  (\hol_k \hol_k\transp) \odot \E_S\big[\one_{I\setminus\{k,\ell\}} \one_{I\setminus\{k,\ell\}}\transp \cdot \indi_I(\ell)\indi_I(k)\big]\\
&  \preceq  \frac{\pi_S(\ell)}{1-\pi_S(\ell)} (\hol_\ell \hol_\ell \transp)\odot \E_{S-1}[ \one_{I} \one_{I}\transp] \\
&\hspace{3cm} + \sum_{k\neq \ell}  \frac{\pi_S(\ell)}{1-\pi_S(\ell)} \left(\hol_k   \frac{\pi_S(k)}{1-\pi_S(k)}  \hol_k\transp \right) \odot \E_{S-2}[ \one_{I} \one_{I}\transp] \\
&\preceq  \frac{\pi_S(\ell)}{1-\pi_S(\ell)} \Big( (\hol_\ell \hol_\ell \transp)\odot \E_{S-1}[ \one_{I} \one_{I}\transp]  +  M \odot \E_{S-2}[ \one_{I} \one_{I}\transp] \Big),
 \end{align*}
 where again $M = \hol \diag(\tfrac{\pi_S}{1-\pi_S}) \hol\transp$.
Finally, applying again Theorem~\ref{th_poisson} \eqref{th_poisson_odot} and~\eqref{pi_Sm1_leq} and similar simplifications as in the proof of (b) yield
 \begin{align*}
\| \E_S[ &\weights ^{-1} \mathbb{I}_{\ell^c} R_I \transp \hol_{I,I} \hol_{I,I}\transp R_I \mathbb{I}_{\ell^c} \weights ^{-1} \cdot  \indi_I(\ell)] \|\\
& \leq  \frac{\pi_\ell}{1-\pi_\ell} \left(\| (\weights ^{-1}\hol_\ell \hol_\ell \transp \weights ^{-1})\odot \E_{S-1}[ \one_{I} \one_{I}\transp] \|  +  \| (\weights ^{-1} M \weights ^{-1}) \odot \E_{S-2}[ \one_{I} \one_{I}\transp] \|\right)\\
& \leq  \frac{\pi_\ell}{1-\pi_\ell} \left(3\| \weights \hol_\ell \hol_\ell\transp \weights\| + \|\diag(\hol_\ell \hol_\ell\transp)\| + 3\| \weights M\weights\| + \|\diag(M)\| \right)\\
& \leq  \frac{\pi_\ell}{1-\pi_\ell} \left(3\| \weights \hol e_\ell \|^2 + \textstyle \max_{k} \hol_{k\ell}^2 + \tfrac{9}{2} \| \weights \hol \weights\|^2 + \tfrac{3}{2} \textstyle \max_k \|e_k\transp \hol\weights\|^2 \right),
 \end{align*}
which together with our assumption that $\|\pi\|_\infty \leq 1/3$ leads to (d). \hfill \textbf{(d) \checkmark} \\

 \noindent \textbf{(\ref{E_RtHHR_indl}-\ref{E_RtNtNR_indl})} Finally, we specialise the general inequality from (d) to the three concrete cases needed in the proofs of Lemma~\ref{lemma_claim1}-\ref{lemma_claim4}.
Reusing the bounds for $H$ from the corresponding special case of (b) as well as $H^2_{k\ell} = \absip{\patom_k}{\patom_\ell}^2\leq \mu(\pdico)^2$ and the assumption $\| \pdico \weights\| \leq 1/3$, yields
 \begin{align}
\| \E [ \weights ^{-1} &\mathbb{I}_{\ell^c} R_I \transp H_{I,I} H_{I,I}\transp R_I\mathbb{I}_{\ell^c} \weights ^{-1}  \cdot \indi_I(\ell)]\| \notag \\
     & \leq 3/2 \cdot \pi_\ell \cdot \left(3\cdot\| \pdico \weights\|^2 + \mu(\pdico)^2 + 9/2 \cdot 1/9 \cdot \| \pdico \weights\|^2  + 3/2 \cdot \| \pdico \weights\|^2 \right) \notag\\
     &\leq  3/2 \cdot \pi_\ell\cdot \max\{\mu(\pdico)^2 , \| \pdico \weights\|^2 \} \cdot 6 \leq 9\cdot \pi_\ell\cdot \max\{\mu(\pdico)^2 , \| \pdico \weights\|^2 \} . \notag
\end{align}
To prove the other two inequalities note that for $\hol = W\transp V$ we get 
\begin{align}
    \| \E [ \weights ^{-1} &\mathbb{I}_{\ell^c} R_I \transp \hol_{I,I} \hol_{I,I}\transp R_I\mathbb{I}_{\ell^c} \weights ^{-1}  \cdot \indi_I(\ell)]\| \cdot \frac{1-\pi_\ell}{\pi_\ell}  \notag \\
     & \leq 3\| \weights W\transp v_\ell \|^2 + \max_{k} (w_k\transp v_\ell)^2 + 9\| \weights W\transp V \weights\|^2 + 3 \max_k \|w_k\transp V \weights\|^2  \notag \\
     & \leq \big (\max_{k} \| w_k\|^2 + 3\|  W\weights\|^2\big ) \cdot \big(\max_\ell \|v_\ell \|^2 + 3 \| V \weights\|^2\big).\label{sym_bound}
\end{align}
Applying this to $\hol = (\pdico \holdiag - Z)\transp \pdico$ and reusing the bounds for the corresponding special case of (b) with $\|\pdico \weights\| \leq 1/8$ yields 
\begin{align}
    \| \E [ &\weights ^{-1} \mathbb{I}_{\ell^c} R_I \transp \hol_{I,I} \hol_{I,I}\transp R_I\mathbb{I}_{\ell^c} \weights ^{-1}  \cdot \indi_I(\ell)]\|  \notag \\     
    & \leq 3/2 \cdot \pi_\ell \cdot \big[ 3\cdot (\|\pdico \weights\| \cdot \eps^2/2 + \|Z\weights\|)^2 +\eps^2\big]\cdot \big(1 + 3 \cdot \| \pdico \weights\|^2\big)\notag \\
    &\leq 3/2 \cdot \pi_\ell\cdot \delta^2 \cdot \big[ 3\cdot (\|\pdico \weights\| \cdot \eps/2 + 1)^2 + 1\big]\cdot \big(1 + 3 \cdot \| \pdico \weights\|^2\big)\leq 9 \cdot \pi_\ell \cdot \delta^2.\notag
\end{align}
Due to the symmetry of \eqref{sym_bound} this also proves the case of $\hol\transp$.
\hfill \textbf{(\ref{E_RtHHR_indl}-\ref{E_RtNtNR_indl})\checkmark} \\
\phantom{bla}
\end{proof}

\bibliography{karinbibtex_3}

\end{document}